\documentclass[12pt,reqno]{amsart}

\usepackage{microtype}
\usepackage{amssymb,bm,mathrsfs}
\usepackage{mathabx} 
\usepackage{enumerate}
\usepackage{tikz}
\usepackage[centering,width=6in]{geometry}
\usepackage[colorlinks,pdfstartview=FitH]{hyperref}

\hypersetup{
    colorlinks,
    citecolor=blue,
    filecolor=black,
    linkcolor=red,
    urlcolor=black
}

\theoremstyle{plain}
\newtheorem{theorem}{Theorem}[section]
\newtheorem{proposition}[theorem]{Proposition}
\newtheorem{lemma}[theorem]{Lemma}

\theoremstyle{definition}
\newtheorem{definition}[theorem]{Definition}
\newtheorem{example}[theorem]{Example}

\newtheorem{rem}[theorem]{Remark}
\numberwithin{equation}{section}

\newcommand{\Tr}{\operatorname{Tr}}

\newcommand{\B}{\mathcal{B}}

\newcommand{\Fp}{\mathbb F_p}
\newcommand{\gb}{\mathbf{g}}
\newcommand{\hb}{\mathbf{h}}

\newcommand{\eb}{\mathbf{e}}
\newcommand{\fb}{\mathbf{f}}

\newcommand{\bb}{\mathbf{b}}
\newcommand{\N}{\mathbb{N}}
\newcommand{\F}{\mathbb{F}}
\newcommand{\Z}{\mathbb{Z}}
\newcommand{\R}{\mathbb{R}}
\newcommand{\C}{\mathbb{C}}

\newcommand{\gvec}{\mathbf g}
\newcommand{\hvec}{\mathbf h}

\newcommand{\kvec}{\mathbf k}
\newcommand{\lb}{\boldsymbol{\ell}}
\newcommand{\abs}[1]{\left|#1\right|}
\newcommand{\norm}[1]{\left\lVert#1\right\rVert}
\newcommand{\Conv}{\mathop{\scalebox{1.5}{\raisebox{-0.2ex}{$\Asterisk$}}}}%
\newcommand{\wh}{\widehat}

\DeclareFontFamily{U}{mathx}{}
\DeclareFontShape{U}{mathx}{m}{n}{<-> mathx10}{}
\DeclareSymbolFont{mathx}{U}{mathx}{m}{n}
\DeclareMathAccent{\widehat}{0}{mathx}{"70}
\DeclareMathAccent{\widecheck}{0}{mathx}{"71}

\usepackage{graphicx} 

\title[A new type of deterministic Salem sets and their spectrality]{A new type of deterministic Salem sets \\ and their spectrality}

\author{Chun-Kit Lai}
\address{San Francisco State University Department of Mathematics, 1600 Holloway Ave, San Francisco, CA 94132
}
\curraddr{}
\email{cklai@sfsu.edu}
\thanks{}

\author{Ruxi Shi}
\address{Shanghai Center for Mathematical Sciences, Fudan University, 200438 Shanghai, China}
\email{ruxishi@fudan.edu.cn}

\author{Yu-Hao Xie}
\address{Department of Stochastics, Institute of Mathematics, Budapest University of Technology and Economics, M\H{u}egyetem rkp.~3., H-1111 Budapest, Hungary}
\email{123yh.xie@gmail.com}
\begin{document}

\keywords{Deterministic Salem sets, Spectral measures, Moran sets}

\begin{abstract}
For all $0< s \le  d$, we provide a new deterministic construction of Cantor sets in arbitrary dimension whose Fourier dimension and Hausdorff dimension are both equal to $s$. The construction is based on a straightforward Cantor-Moran construction with contraction ratios given by reciprocals of integers. The key tool to obtain the fast Fourier decay is due to the Weil bound in analytic number theory. Furthermore,  we show that the  natural equal-weighted  Cantor-Moran measure is the desired measure admitting the near optimal Fourier decay and the measure admits an exponential orthonormal  basis $\{e^{2\pi i \lambda \cdot x}: \lambda\in \Lambda\}$ for its $L^2$ space. This gives the first examples of  singular Salem spectral measures in $\R^d$.
\end{abstract}

\maketitle
\tableofcontents

\section{Introduction}

Let $\mu$ be a non-trivial finite Borel measure in $ \R^d $. The Fourier dimension of $\mu$ is defined to be 
\[
\dim_F\mu=\sup\left\{t\in[0,d]:\exists 0<C<\infty\text{ such that }|\widehat{\mu}(\xi)|\le C(1+|\xi|)^{-t/2} \textup{ for all }  \xi\in\R^d \right\}.
\]
For a Borel set $E\subset\R^d$, the Fourier dimension of $E$ is defined to be
\[
\dim_F E=\sup\{\dim_F\mu:\mu\text{ is a probability measure supported on }E\}.
\]
It is well-known that $\dim_FE\le \dim_HE$, where $\dim_HE$ is the Hausdorff dimension of $E$. The Borel set $ E $ is called a {\bf Salem set} if $\dim_FE =\dim_HE$ and $\mu$ is called a {\bf Salem measure} if $\dim_F\mu$ is equal to the Hausdorff dimension of its closed support. Salem sets are ubiquitous in probabilistic models, where readers can refer to \cite{LabaPramanik2009}, \cite{Kahane1985RandomSeries}, \cite{Shmerkin2017}, or Salem's original paper \cite{Salem1951} for more details, but they are very difficult to construct deterministically. Currently, the only known deterministic construction of Salem sets was due to Kaufman \cite{Kaufman1981} via compact sets inside the Jarn\'ik--Besicovitch sets in Diophantine approximation (see also Wolff's lecture notes \cite{Wolff2003HarmonicAnalysis} and \cite{Bluhm1998Kaufman}). Kaufman's construction was extended to high dimensions by R. Fraser and Hambrook \cite{Hambrook2017ExplicitSalem,FraserHambrook2023} using machinery from algebraic number theory.  

The goal of this paper is to construct, for any given $s\in (0,d]$, a new class of deterministic Salem sets via the standard Moran construction in $\R^d$. Moreover, our measure has an additional feature that all these measures are {\bf spectral measures}, which means that each of them admits an exponential orthonormal basis $\{e^{2\pi i \lambda\cdot x}:\lambda\in \Lambda\}$ for some frequency set $\Lambda$ in its $L^2$ space. Our main result also answers a folklore problem about the existence of spectral measures with Fourier decay.  Here is our main theorem of the paper:

\begin{theorem}\label{thm:main}
For every $0< s\le d$, there is a deterministic homogeneous Moran set $K\subset \R^d$  with its associated Moran measure  $\mu$ such that $\dim_F\mu = \dim_HK = s$ and $\mu$ is a spectral measure. Consequently, $K$ is a Salem set.
\end{theorem}

 Our construction is based only on the Weil estimate over finite fields and the construction is based on only the standard Cantor-Moran set construction. Moreover, at each finite stage of the construction, the finite sets are indeed a kind of the finite Salem sets introduced by Iosevich and Rudnev \cite{IosevichRudnev2007}.  It is believed that it will provide a much more convenient example for problems in which fractal sets with power Fourier decay play an important role (see \cite{Li-Liu-2026-diophantine,fraser2025sharpnessmockenhauptmitsisbakseegerfourierrestriction} and the references therein). 

The study of spectral measures dates back to Fuglede \cite{Fuglede1974} and his celebrated conjecture, now known as Fuglede's conjecture, which asserts that spectral sets (i.e. Lebesgue measure on the set is a spectral measure) and translational tiles are the same class of sets. The conjecture was first disproved by Tao \cite{Tao2004Fuglede}, but it is still an active area of research. In 1998, Jorgensen and Pedersen \cite{JP98} pioneered the study into the realm of fractals and it has been refined into many different settings.   We refer the readers to the recent survey by Kolountzakis for the comprehensive narrative of this conjecture \cite{Kolountzakis2026OrthogonalFourier}. 

 Spectrality of the measures is a surprising outcome of our main result. Apart from Lebesgue measures and some trivial examples of surface measures \cite{li2025fourierframessalemmeasures}, all currently known spectral measures do not have any Fourier decay. It has been a folklore conjecture that no spectral measure of fractional dimension has any Fourier decay. Theorem \ref{thm:main} however disproved this conjecture by finding a Moran-Salem spectral measure deterministically.

\subsection{Setup, Moran sets and Salem measures.} We now present our setup and the main terminology used in the proof.  Let $\{N_n\}_{n=1}^{\infty}$ be a sequence of integers greater than or equal to $2$ and let $\B_n$ be a finite non-empty subset of $\Z^d$. The (integral) {\bf homogeneous Moran set} generated by the sequence of pairs $\{(N_n,\B_n)\}_{n=1}^{\infty}$ is the unique closed set $ K $ such that 
\begin{equation}\label{eq_Moran}
K =K(\{N_n,\B_n\}) = \left\{\sum_{n=1}^{\infty} \frac{d_n}{N_1\cdots N_n}: d_n\in\B_n\right\}.
\end{equation}
We will write $K$ if $N_n, \mathcal{B}_n$ are clear from the context. Technically, $K$ may be unbounded, but we will make it compact by imposing the natural condition
$$
\sum_{n=1}^{\infty} \frac{\max\{|d_n|: d_n\in\B_n\}}{N_1\cdots N_n} <\infty.
$$
Homogeneous Moran sets form a natural generalization of the classical middle-third Cantor set (where $N_n = 3$ and $\B_n = \{0,2\}$ for all $n \ge 1$) and deleted-digit self-similar Cantor sets where all $N_n$ and $\B_n$ are the same for all $n$. Homogeneous Moran sets have a long history of study and they have many applications in different areas of fractal geometry (see e.g. \cite{FengWenWu1997,Feng2012SalemBernoulli}). 

Homogeneous Moran sets have also appeared naturally in harmonic analysis since  Moran sets in (\ref{eq_Moran}) naturally support a \textbf{Cantor-Moran measure} $\mu$ having a readily computable Fourier transform formula. The associated Cantor-Moran measure $\mu$ is defined by the following infinite convolution 
\begin{equation}\label{cantor-moran_measure}
    \mu = \Conv_{ n = 1 }^{ \infty } \left( \sum_{ d \in \mathcal{B}_n } \dfrac{1}{ \# \mathcal{B}_n }\delta_{ d \cdot ( N_1 \cdots N_n )^{-1} } \right),
\end{equation} 
where $ \delta_a $ represents the Dirac measure concentrated at the point $ a\in \mathbb{R}^d $.  Its Fourier transform is given  by an infinite product of trigonometric polynomials. 
 \begin{equation}\label{eq-FT}
 \widehat{\mu} (\xi) = \prod_{n=1}^{\infty} {\mathsf M}_{\B_n} ((N_1\cdots N_n)^{-1}\xi), \ \mbox{where} \ {\mathsf M}_{\B_n} (\xi) = \frac1{\#\B_n}\sum_{d\in\B_n} e^{-2\pi id \cdot \xi}. 
 \end{equation}
Throughout this paper, ${\mathsf M}_{\B_n}(\cdot)$ will be called the {\bf mask function} for $\B_n$. Recall that the Fourier transform of a finite Borel measure \(\mu\) on \(\mathbb R^d\) is defined by
$$
\widehat{\mu}(\xi ) = \int e^{-2\pi i  \xi \cdot x }~d\mu(x), \qquad \xi\in\R^d.
 $$

Shmerkin~\cite{Shmerkin2017} discussed a probabilistic Cantor-type construction resulting almost surely in a Salem set. He showed that if we consider Cantor sets with contraction ratios $N_n^{-1}$ in stage $n$, and select the $(n+1)$-th stage of basic intervals independently across all the $n$-th stage intervals, then the resulting random Cantor set is almost surely Salem. He used this construction to show that there exist Salem sets containing no nontrivial three-term arithmetic progressions. As homogeneous Moran sets have to preserve the same collection of basic intervals in the next stage of construction, this random construction is different from constructing homogeneous Moran sets and the Fourier transform of the measures does not have a convolutional product structure. Another construction was also considered by {\L}aba and Pramanik \cite{LabaPramanik2009}, as well as by Hambrook and {\L}aba \cite{HambrookLaba2013}, in their studies of the sharpness of Fourier restriction estimates under Mockenhaupt's condition. Their construction was not based on a Moran construction, however, we will see that their arguments are more closely related to our work and can be used to construct a probabilistic Moran-Salem set (see Remark \ref{rem3.6}).  

The convolution structure of homogeneous Moran measures seems to make
Fourier decay difficult to establish. For example, it is well-known that the middle-third Cantor set has Fourier dimension zero. Taking $N_n = 3$ and $\B_n = \{0,2\}$ for every $ n \in \mathbb{N} $, (\ref{eq-FT}) becomes 
$$\widehat{\mu}(\xi) = e^{ - \pi i \xi}\prod_{n=1}^{\infty}\cos(2\pi \xi 3^{-n}),$$
for $ \xi \in \mathbb{R} $. It follows directly from this formula that $|\widehat{\mu}(3^n)| = |\widehat{\mu}(1)|\ne 0$ for all $ n \in \mathbb{N} $. A closed set $E\subset [0,1]$ is a \textbf{set of uniqueness},
or a \emph{$U$-set}, if for any sequence of complex numbers $ \{ c_k \}_{k\in \mathbb{Z}} $,
\[
  \lim_{m\to\infty}\sum_{k=-m}^{m}c_ke^{2\pi ikx}=0, \ \textup{ for every } x\in [0,1]\setminus E
\]
implies that $ c_k = 0 $ for every $ k \in \mathbb{Z} $.
 It is known that if $E$ is a set of uniqueness, then $E$ supports no measure that has Fourier decay and hence $\dim_FE=0$ \cite{KechrisLouveau1987}. The middle-third Cantor set is a set of uniqueness.  The behavior generalizes easily to a large class of homogeneous Moran sets.  Indeed, to the best of our knowledge, all currently known homogeneous Moran sets of zero Lebesgue measure are sets of uniqueness, and hence they have Fourier dimension zero \cite{Lai2017,Lai-Xie-2026}. An important and attractive question here is the following:

{\bf (Qu 1):} Do we have any Lebesgue measure zero homogeneous Moran sets  that have positive Fourier dimension? If so, are they Salem? Can they be constructed deterministically?  

It was even conjectured in \cite[Section 7]{Lai-Xie-2026} that all Moran sets of Lebesgue measure zero are sets of uniqueness. But we are going to disprove our conjecture and show, rather surprisingly, that the opposite of the conjecture is true and {\bf (Qu 1)} has a positive answer.

\subsection{Spectral measures} A Borel probability measure $\mu$ on $\R^d$ is called a {\bf spectral measure} if there exists a countable set $\Lambda\subset \R^d$ such that  $\mu$ admits an exponential orthonormal basis $\{e^{2\pi i \lambda\cdot x}\}_{\lambda\in \Lambda}$ for $L^2(\mu)$, and $\Lambda$ is called a {\bf spectrum} for $\mu$. 


 Jorgensen and Pedersen \cite{JP98} discovered that the middle-fourth Cantor measure is a spectral measure in 1998. The mechanism for Jorgensen and Pedersen can be described by Hadamard triples. Let \(N\ge2\) be an integer, and let \(B,L\subset\mathbb Z^d\) be finite nonempty sets with \(|B|=|L|\). We say that  $(N, B, L)$ is a {\bf Hadamard triple} if the matrix 
 $$
 \frac{1}{\sqrt{|B|}} \left(e^{2\pi i \frac{b\cdot \ell}{N}}\right)_{\ell\in L, b\in B}
 $$
 is a unitary matrix. It also means that the equal-weighted discrete measure $\delta_B$ admits a spectrum $\frac{1}{N}L$.  Suppose that $\{(N_n,B_n,L_n)\}$ is a tower of Hadamard triples.  Then the convolution of the first $n$ factors in \eqref{cantor-moran_measure} is a spectral measure with spectrum $L_1+N_1L_2+\cdots+(N_1\cdots N_{n-1})L_n$. However, it is not easy to show that the limiting Cantor-Moran measure is a spectral measure. The case of the self-similar measures in $\R$ was first studied by {\L}aba and Wang \cite{LW02} and finally its higher dimensional generalization was completed by Dutkay, Haussermann and Lai \cite{DHL2019}. For the general Moran measures, Strichartz \cite{Strichartz2000MockFourier} gave a very first study. An and He~\cite{AnHe2014} subsequently identified a large class of spectral Moran measures. A general theory was  obtained in \cite{AnFuLai2019} and \cite{Li-Miao-Wang}, among many other authors studying different aspects of the problem (see e.g. \cite{LLZ2026} and references therein). 

Despite many studies about how these Hadamard triple systems can generate spectral Cantor-Moran measures, no such {singular} measure with Fourier decay was previously known. More generally, even without assuming that the measure arises from a Hadamard triple construction, the following question has remained open:

{\bf (Qu 2):} Does there exist any singular spectral measure of fractional dimensions in $\R^d$ that has Fourier decay? 

In the deterministic case, it has been known that surface measures on the convex body of positive Gaussian curvature, which are Salem, admit no Fourier frame and hence are non-spectral \cite{IosevichLaiLiuWyman2022} (see also \cite{ChenLiu2025FourierFrames}).  Under various natural probabilistic models, Li and Liu \cite{li2025fourierframessalemmeasures} demonstrated the non-existence of Fourier frame (hence exponential orthonormal basis) in these models with high probability. However, in the same paper, they also remark that the push forward of the standard Lebesgue measure of $[-1/2,1/2]$ onto the spherical arc in the unit circle in $\R^2$  is a spectral measure, while it is Salem. This shows that a Salem spectral measure can exist trivially for sets with Hausdorff dimension one in $\R^2$. This leaves whether the existence of  Salem spectral measure with a fractional Hausdorff dimension an interesting, but delicate question. 

\subsection{Outline of the proof.} Theorem \ref{thm:main} settles  both {\bf (Qu 1)} and {\bf (Qu 2)} affirmatively in any non-integral Hausdorff dimension in $\R^d$.

The following is  a brief outline of the proof.  The precise Moran construction will be described in Section \ref{Sec4} and will be stated in Theorem \ref{main-theorem-detail} for $\R$ and in Section \ref{Sec8} in Theorem \ref{main-theorem-detail2} for $\R^d$, $d\ge 2$. On $\R$,  the main idea of achieving the near optimal Fourier decay is to construct digit sets $\B_n$ inside $\{0,1,\cdots, N_n-1\}$ such that the mask function ${\mathsf M}_{\B_n}(\cdot)$ is close to the Dirichlet mask function $ \mathsf{D}_{N_n}(\cdot) $, which is the mask for $\{0,1,\cdots, N_n-1\}$, in the supremum norm of $\R$. A very closely related result is in \cite[Lemma 6.2]{LabaPramanik2009} where such $\B_n$ can be constructed probabilistically. They attributed the idea to Green \cite{Green2002ArithmeticProgressions}.  The same lemma also appeared in \cite[Lemma 7]{HambrookLaba2013}. 

We will first focus on the construction in $\R^1$ as this constitutes the clearest layout for the whole proof. In order to construct deterministically the required Moran measures, we need to construct such $\B_n$ deterministically. It turned out that this can be done via the Weil bound over finite fields. 
The technique of Weil estimates goes back to Gauss's work on quadratic reciprocity \cite{Gauss1811}.
For an odd prime $p$, his evaluation of the quadratic Gauss sum gives
\[
\sum_{x=0}^{p-1}e^{2\pi i x^2/p} =
\begin{cases}
\sqrt p, & p\equiv1\pmod4,\\
i\sqrt p, & p\equiv3\pmod4.
\end{cases}
\]
Weil's bound extends this phenomenon to polynomial
additive character sums over finite fields \cite{Weil1948};
the version used in the paper is stated in Theorem~\ref{thm:weil}. At a single level in our construction, we construct a digit set $ \mathcal{B} $ containing \(p^r\) many elements inside \(\{0,1,\ldots,p^q-1\}\) where $r<q$. Fixing a basis of \(\mathbb F_{p^r}\) over \(\mathbb F_p\), we represent each element in $ \mathbb{F}_{p^r} $ by \(r\) coordinates in \(\mathbb F_p\). For each \(x\in\mathbb F_{p^r}\), concatenate the coordinate vectors of \(x,x^2,x^3,\ldots\), and let \(g_0(x),\ldots,g_{q-1}(x)\) be the first \(q\) entries, represented as integers in \(\{0,\ldots,p-1\}\). Define the injective map $ b:\mathbb F_{p^r} \to \{0,\ldots,p^q-1\} $ by 
$$ \mathsf{b}(x)=\sum_{j=0}^{q-1}g_j(x)p^j $$
(see \eqref{eq:bdef}). Indeed, the set of vectors $$\{(g_0(x),\cdots, g_{q-1}(x)): x\in\F_{p^r}\}$$ is a finite Salem set introduced by Iosevich and Rudnev \cite{IosevichRudnev2007} in the finite vector space $\F_p^q$.  Projecting the set into $\N$ via $p$-adic expansions, we define the digit set $$ \mathcal B=\{ \mathsf{b}(x):x\in\mathbb F_{p^r}\} \subset \{0, 1, \cdots, p^q-1\}.$$

{\color{black}
The idea of this digit construction was first suggested by ChatGPT-5.6Sol. After some literature search, we discovered that it was not new. Indeed, it was first used in numerical integration. When $r=1$, Korobov~\cite{Korobov1967} used powers of finite-field elements and Weil's bound to approximate averages over finite grids in numerical analysis.  For $r\ge2$, a closely related construction over finite extension fields was used by Dick~\cite[Section~7.1]{Dick2014} on the numerical approximation of high-dimensional integrals.}

We will compute the Fourier transform of $\B$ via the finite Fourier inversion  formula and expand it into polynomial character sums in Proposition \ref{thm-weil}. The zero-frequency term is exactly the Dirichlet mask, while Weil's bound controls the remaining terms. This yields a uniform approximation of the digit mask by the Dirichlet mask 
\begin{equation}\label{d-esti}
    \sup_{\xi\in\mathbb R} \left| \mathsf M_{\mathcal B}(\xi)-\mathsf D_{p^q}(\xi)
\right| \le (D-1)p^{-r/2}\bigl(2\log(2p)\bigr)^q,
\end{equation}
where $ D = \lceil q/r\rceil$ and $p>D$ is an odd prime. The precise execution will be developed in Theorem \ref{thm3.1-Weil} in Section \ref{Sec3}. In Section \ref{Sec2}, we will provide the preliminaries on the finite field needed to set up our proof in Section \ref{Sec3}. 

To construct the required Moran measure, we will provide an iteration scheme to recursively define prime numbers $p_n$ and the associated digit sets $\B_n$ satisfying Theorem \ref{thm3.1-Weil}. The one-dimensional case will be given in Section \ref{Sec4}.  Once a frequency $\xi$ is fixed, the estimate for Fourier transform of Moran measures is completely determined by three levels of the mask functions whose scales are compatible with $\xi$. The Weil estimates at these levels provide the necessary Fourier decay. 
This proof will be carried out in Section \ref{Sec5}.

 After we settle the case for $\R$, we can adopt a similar strategy for higher dimensions. However, the estimate in (\ref{d-esti}) falls short if we want to construct a high dimensional Salem set with dimension $>2$. We will instead enlarge our choice of the digits to $\{0,1, \cdots, H(p^q-1)\}^d$ so that we can approximate the $H^{\rm th}$ power of the standard Dirichlet mask.   
\begin{equation}\label{d-esti-high}
    \sup_{\xi\in\mathbb R^d} \left| \mathsf M_{\mathcal B}(\xi)-\prod_{j=1}^d (\mathsf D_{p^q}(\xi_j))^H
\right| \le (D-1)p^{-dr/2}\bigl(2\log(2p^r)\bigr)^d (2\log (2p))^{Hdq},
\end{equation}
where $D = \lceil Hq/r\rceil$.  Such an estimate can be done and this Weil estimate will be proved in Theorem \ref{thm:finite} in Section \ref{Sec7}. Once this estimate is complete, the same idea of proof for the one-dimensional construction will carry through to arbitrary dimensions. We will prove this in detail in Section \ref{Sec8}. Although the introduction of $H$ will make the Cantor-Moran set supported on $[0,H]^d$, the set is still compactly supported and we can rescale the set and measures by $H^{-1}$ if we want a set lying inside the standard unit cube.

Finally, we will see that all $\B_n$ are spectral sets in their corresponding finite group. The theory of spectral Moran measures for this spectral measure developed in \cite{LaiWang,AnFuLai2019,LiWang2024} still applies. Therefore,  the resulting measure is a spectral measure. The detailed proof will be given in Section \ref{Sec9}.

\subsection{Disclosure of AI use.} Our initial approach to Theorem \ref{thm:main} was inspired by the Rudin–Shapiro polynomial construction (see e.g. \cite[p.34]{Katznelson2004HarmonicAnalysis}). We input this basic idea into ChatGPT 5.6-Sol, and it was confirmed that the measure is spectral and has a positive Fourier dimension, though it is not Salem. With this encouraging outcome, we continued to use AI to explore possible refinements. After many exchanges, AI proposed to use the Weil bound instead of refining the Rudin–Shapiro polynomials.  The authors checked carefully every step it produced, rewrote the whole proof, expanded all unclear details and added in possible motivations from related work. We take full responsibility for all mathematical claims, proofs, and the final content of the paper. 


\subsection{Acknowledgement.} Part of the work was initiated when Chun-Kit Lai and Yu-Hao Xie were residing in The Chinese University of Hong Kong. The authors would like to thank Professor De-Jun Feng for his hospitality. Chun-Kit Lai was partially supported by the AMS-Simons Research Enhancement Grants for Primarily Undergraduate Institution (PUI) Faculty. Ruxi Shi was supported by the NSFC Nos.~12571198 and 12231013.

\section{Preliminaries on Finite Fields}\label{Sec2}
In this paper,  $p$ is an odd prime and we will fix $\F_p$ to be the finite field of $p$ elements. We will identify $\F_p = \{0,1,\cdots, p-1\}$ with arithmetic performed modulo $p$. $\Fp^r = \Fp\times\cdots\times \Fp$ ($r$ times) denotes the $r$-dimensional vector space over $\F_p$.  Finally,  $\F_{p^r}$ will denote the (unique) finite field of $p^r$ elements.  From modern algebra,  $\F_{p^r}$ is the splitting field of the polynomial $x^{p^r}-x$ over $\F_p$. We can write $\F_{p^r} = \F_p(\alpha)$, where $\alpha$ is a root of a degree $r$ irreducible polynomial over $\F_p$.

When an $\alpha\in\F_{p^r}$ whose minimal polynomial has degree $r$ is chosen,   $\{1,\alpha, \cdots, \alpha^{r-1}\}$ is a basis for $\F_{p^r}$ over $\F_p$. Hence, for all $x\in \F_{p^r}$, there exists a unique $(c_0(x),\cdots, c_{r-1}(x))\in \F_{p}\times\cdots\times \F_{p}$ ($r$ times) such that 
\begin{equation}\label{eq:digit-expansion}
x = c_0(x)+ c_1(x)\alpha+\cdots+c_{r-1}(x)\alpha^{r-1}.
\end{equation}
Furthermore, the coordinate map
\begin{equation}\label{isomorphism}
C:\F_{p^r}\longrightarrow\F_p^r,
\qquad
C(x)=\bigl(c_0(x),\ldots,c_{r-1}(x)\bigr),
 \end{equation}
 is an $\F_p$-isomorphism between $\F_{p^r}$ and $\F_{p}^r$ ($r$ copies of $\F_p$ in a Cartesian product) where $\F_{p}^r$ is considered as a vector space over $\F_p$ under its addition and $\F_p$-scalar multiplication.

\subsection{Trace and linear functionals} We study the $\F_{p}$-linear functionals of $\F_{p^r}$ and provide a characterization by its  trace function. Readers can refer to \cite{Finitefield-book} or \cite{IrelandRosen1990} for more details.   Define the {\bf trace} of $\F_{p^r}$ by $\Tr: \F_{p^r}\to \F_p$ 
$$
\Tr(x) = x+x^p+\cdots + x^{p^{r-1}}.
$$
In the literature, people usually write $\Tr_{{\mathbb F}_{p^r}/\F_p}(x)$. But we omit the field dependence since we will only consider one field in the next two sections while studying trace. 
We begin with a lemma about the property of the trace map. This can be deduced easily using basic properties of finite fields.
\begin{lemma}\label{lem:trace-property}\cite[Theorem 2.23]{Finitefield-book}
$\Tr$ is a well-defined linear transformation from $\F_{p^r}$ to $\F_p$ that maps onto $\F_{p}$.  
\end{lemma}
We provide a complete characterization for the linear functionals of $\F_{p^r}$. This can also be found in \cite[Theorem 2.24]{Finitefield-book}. We provide a proof for completeness. 
\begin{lemma}\label{lem:trace}
For every $\F_p$-linear functional $\ell:\F_{p^r}\to\F_p$, there exists a unique $A\in {\mathbb F}_{p^r}$ such that 
\[
\ell(x)=\Tr(Ax).
\]

\end{lemma}

\begin{proof}
Let $ F^* $ be the dual space of $\F_{p^r}$ over $\F_p$. By Lemma \ref{lem:trace-property}, the map $\varphi_A: x\mapsto\Tr(Ax)$ is in $F^{\ast}$ for all $A\in \F_{p^r}$. Define the map $\Phi: \F_{p^r}\to F^{\ast}$ by
$$
\Phi(A) = \varphi_A.
$$
A routine check shows that $\Phi$ is a linear map over $\F_p$. The lemma will follow if we can show that $\Phi$ is bijective. Since both spaces $F^{\ast}$ and $\F_{p^r}$ have dimension $r$ over $\F_p$, by the rank-nullity theorem, it suffices to show that $\Phi$ is injective.

Suppose that  $\Phi$ is not injective. Then there exists $A\ne 0$ such that  $\varphi_A(x) = 0$ for all $x\in \F_{p^r}$. By Lemma \ref{lem:trace-property}, we can find $a\ne 0$, such that $\Tr(a)\ne 0$. This implies that $\varphi_{A}(A^{-1}a) = \Tr(a)\ne0$, which is a contradiction. Hence, $\Phi$ is injective and therefore bijective, proving the lemma.
\end{proof}

\subsection{Weil bound for polynomial additive character sums}
We first recall that an \emph{additive character} of $\F_{p^r}$  is an additive homomorphism $\psi:\F_{p^r} \to \{z\in\C:|z|=1\}$, i.e. 
\[
\psi(x+y)=\psi(x)\psi(y), ~~\forall x,y\in \F_{p^r}.
\]
It is called \emph{nontrivial} if $\psi\not\equiv1$,
that is, if there exists $x\in\F_{p^r}$ such that $\psi(x)\ne1$.

The following classical theorem is the key non-elementary finite-field estimate used in the construction (see e.g. \cite[Theorem 5.38]{Finitefield-book}).

\begin{theorem}[Weil bound]\label{thm:weil}
Let $ r \ge 1 $ and $ p $ be a prime. Let $\psi:\F_{p^r}\to\C$ be a nontrivial additive character, and let
\[
P(X)=A_1X+\cdots+A_eX^e\in \F_{p^r}[X]
\]
be a polynomial of degree $e\ge2$. If $p>e$, then
\begin{equation}\label{weil-bound}
    \left|\sum_{x\in \F_{p^r}}\psi(P(x))\right|\le(e-1)p^{r/2}.
\end{equation}
\end{theorem}

For $ x \in \mathbb{R} $, let $\lfloor x\rfloor$ denote the greatest integer less than or equal to $x$, and $\lceil x\rceil$ denote the least integer greater than or equal to $x$. 

Let $C$ be the map in (\ref{isomorphism}).  Let $1 \le r<q$ and $p$ be an odd prime such that $p>D$ where $D = \left\lceil\frac{q}{r}\right\rceil$.  Concatenate the coordinate vectors
of $x,x^2,\ldots,x^D$ and retain the first $q$ entries. This defines $\Psi_{r,q}: \F_{p^r} \longrightarrow \Fp^q$ by
\begin{equation}\label{eq:Psi}
 \Psi_{r,q}(x)=\big(C(x),C(x^2),\ldots,C(x^D)\big)\big|_{\text{first }q\text{ entries}}\in\Fp^q.
\end{equation}
Thus the first $r$ entries of $\Psi_{r,q}(x)$ are exactly $C(x)$. 
The last block is truncated if $r$ does not divide $q$.

Let us recall basic Fourier analysis on the vector space $\Fp^q$. 
Write 
$$e_p(t)=e^{2\pi it/p}$$
and for 
$F:\Fp^q\to\mathbb C$, we define the Fourier transform of $F$ as 
$$ \widehat F(v)=p^{-q}\sum_{y\in\Fp^q}F(y)e_p(-v\cdot y).   
$$
We have the Fourier inversion formula:
\begin{equation}\label{Fourier-inversion}
F(y) = \sum_{v\in\Fp^q} \widehat{F}(v)e_p(v\cdot y), ~\forall y\in\F_p^q.
\end{equation}
We define the {\it Wiener norm} for $F$ to be 
$$
\|F\|_{W} = \sum_{v\in \F_p^q} |\widehat{F}(v)|.
$$
{We use $ \mathbf 0 $ to denote the zero element in $ {\mathbb F}_p^q $.} 

\begin{definition} (Iosevich and Rudnev \cite{IosevichRudnev2007})
  A finite set $S\subset \F_p^q$ is called a {\bf (finite) Salem set} if {\color{black}there exists a constant $ C > 0 $ independent of $ p $ such that}
$$
\sup_{\color{black}\xi \in \F_{p}^q \setminus \{ {\bf 0} \}}|\wh{{\mathbf 1}_S}(\xi)|\le C \cdot p^{-q} (\#S)^{1/2}.
$$

\end{definition}
Further discussions about finite field Salem sets can be found in the recent paper by J. Fraser \cite{Fraser2024}. The following proposition shows that the set $S:=\Psi_{r,q}(\F_{p^r})$ is a Salem set in the above sense. Furthermore, we show that all functions on $\Fp^q$ obey a Weil type estimate. 
{\color{black} The averaging estimate in part~(2) of the following proposition first appeared in the finite Fourier argument used by Korobov~\cite[proof of Theorem~1]{Korobov1967}.}

\begin{proposition}\label{thm-weil}
  Let $1\le r<q$ and $p$ be an odd prime such that $p>D = \lceil q/r\rceil$. 
  \begin{enumerate}
  \item The set $\Psi_{r,q}(\F_{p^r}) = \{\Psi_{r,q}(x): x\in\F_{p^r}\}$  is a finite Salem set in $\Fp^q$ with 
 $$
 |\wh{{\mathbf 1}_S}(v)|\le (D-1) p^{-q} (\# S)^{1/2}
 $$
 for all $v\ne{\bf 0}.$
 \item For every function $F:{\mathbb F}_p^q\to\mathbb C$,
\begin{equation}\label{eq:averaging}
 \left|
 p^{-r}\sum_{x\in\mathbb F_{p^r}}F(\Psi_{r,q}(x))
 -p^{-q}\sum_{y\in\Fp^q}F(y)
 \right|\le(D-1)p^{-r/2}\cdot \|F\|_{W}.
\end{equation}
 \end{enumerate}
\end{proposition}
\begin{proof}
(1). {\color{black}Fix $v=(v_0,\ldots,v_{q-1})\in\Fp^q\setminus\{\mathbf 0\}$.
For notational convenience, extend the coordinate sequence $v_0,\ldots,v_{q-1}$ to $v_0,\ldots,v_{Dr-1}$ by setting
$$ v_j:=0,\qquad q\le j<Dr. $$
These additional coefficients correspond to the coordinates of $(C(x),C(x^2),\ldots,C(x^D))$ that are omitted in the definition of $\Psi_{r,q}$. For each $1\le k\le D$, define $ \varphi_k:\mathbb F_{p^r}\longrightarrow\Fp $ by
$$ \varphi_k(z)=\sum_{j=0}^{r-1}v_{(k-1)r+j}c_j(z). $$
Since each coordinate map $c_j(\cdot)$ is $\Fp$-linear, so is $\varphi_k (\cdot)$. By the definition of $\Psi_{r,q}$, for every $x\in\mathbb F_{p^r}$,
$$ v\cdot\Psi_{r,q}(x) =\sum_{k=1}^{D}\sum_{j=0}^{r-1} v_{(k-1)r+j}c_j(x^k) =\sum_{k=1}^{D}\varphi_k(x^k).
$$

Since $v\ne\mathbf 0$, there exist $1\le k_0\le D$ and $0\le j_0<r$ such that $v_{(k_0-1)r+j_0}\ne0$. By \eqref{eq:digit-expansion}, $c_j(\alpha^{j_0})$ equals $1$ when $j=j_0$ and $0$ otherwise. Hence
$$ \varphi_{k_0}(\alpha^{j_0}) =\sum_{j=0}^{r-1}v_{(k_0-1)r+j}c_j(\alpha^{j_0}) =v_{(k_0-1)r+j_0}\ne0. $$
Thus $\varphi_{k_0}$ is a  non-zero functional.}

By Lemma \ref{lem:trace}, for each $1\le k \le D$, there is a unique $A_k\in\mathbb F_{p^r}$ satisfying $\varphi_k(z)=\Tr(A_kz)$.  For the index $k_0$ chosen above, it follows that $A_{k_0}\ne0$; otherwise, $\Tr(A_{k_0}z)=0$ for every $z$, contradicting that $\varphi_{k_0}\ne0$. Let $e = e(v)\ge 1$ be the
largest index with $A_e\ne0$. Then $ 1\le e\le D<p $. Let $ P_v(X)=A_1X+\cdots+A_eX^e $.  It follows
$$ v\cdot\Psi_{r,q}(x) = \sum_{k=1}^{D}\varphi_k(x^k) = \sum_{k=1}^D \Tr(A_k x^k)  =  \Tr\big(P_v(x)\big). $$
If $e=1$, then $P_v(X) = A_1X$ and 
$$
p^{-r}\sum_{x\in\F_{p^r}}e_p(v\cdot\Psi_{r,q}(x)) = p^{-r}\sum_{x\in\F_{p^r}}e_p(\Tr (A_1x)) = 0
$$
since $A_1\ne 0$ and $e_p(\Tr(A_1x))$ is a non-trivial character over $\F_{p^r}$. It is well known that the sum of a non-trivial character over all elements in the additive group is zero. 
For $2\le e<p$, by the polynomial additive-character Weil bound in Theorem \ref{thm:weil},
\begin{equation}\label{eq:Salem}
 p^{-r}\left|\sum_{x\in\F_{p^r}}e_p(v\cdot\Psi_{r,q}(x))\right| = p^{-r}\left|\sum_{x\in\F_{p^r}}e_p(\Tr(P_v(x)))\right|
 \le(e-1)p^{-r/2}\le(D-1)p^{-r/2}.
\end{equation}
Rearranging the terms and noting that $\#\Psi_{r,q}(\F_{p^r}) = p^r$, we have 
$$
|\wh{{\mathbf 1}_S}(v)| = p^{-q} \left|\sum_{x\in\F_{p^r}}e_p(v\cdot\Psi_{r,q}(x))\right|\le (D-1) p^{-q} (\# S)^{1/2}.
$$
This completes the proof of (1). 

\medskip

(2). We first apply the  Fourier inversion formula (\ref{Fourier-inversion}) and interchange the sums to obtain
\begin{equation}\label{do-Fourier-inversion}
    p^{-r}\sum_{x\in\mathbb F_{p^r}}F(\Psi_{r,q}(x)) =\sum_{v\in\Fp^q}\widehat F(v)  \left(p^{-r}\sum_{x\in\mathbb F_{p^r}} e_p(v\cdot\Psi_{r,q}(x))\right).
\end{equation}
Recall that $ \mathbf 0 $ is the zero element in $ {\mathbb F}_p^q $. When  $v=\mathbf 0$, 
\begin{equation}\label{zero-coefficient}
    \widehat{F}(\mathbf 0) = p^{-q}\sum_{y\in\Fp^q}F(y).
\end{equation}
Subtracting $\widehat F(\mathbf0)$ from both sides of \eqref{do-Fourier-inversion} and using \eqref{zero-coefficient}, we obtain
\begin{equation}\label{eq:Sp}
 \begin{aligned}
 &p^{-r}\sum_{x \in\mathbb F_{p^r}} F(\Psi_{r,q}(x))-p^{-q}\sum_{y \in\mathbb F_{p}^q} F(y)=\sum_{v\ne \mathbf 0}\widehat F(v)
       \left(p^{-r}\sum_{x\in\F_{p^r}}e_p(v\cdot\Psi_{r,q}(x))\right).
 \end{aligned}
\end{equation}
Taking absolute values, applying \eqref{eq:Salem} to each term,
and then summing gives
\[
 \begin{aligned}
 \left|p^{-r}\sum_{x\in \F_{p^r}}F(\Psi_{r,q}(x))-p^{-q}\sum_{y\in\F_p^q}F(y)\right|
 &\le(D-1)p^{-r/2}\sum_{v\ne 0}|\widehat F(v)|\\
 \end{aligned}
\]
This proves \eqref{eq:averaging}.
\end{proof}

\begin{rem}
 {\rm    Assume that $p$ is an odd prime and that $\psi$ is a nontrivial additive character of $\F_{p^r}$. We remark that the upper bound $p^{r/2}$ in \eqref{weil-bound} is sharp. Indeed, if we take $P(x) = Ax^2$ for $ A \neq 0 $, define $G(A) = \sum_{x\in \F_{p^r}}\psi(Ax^2)$. Since $\overline{\psi(z)}=\psi(-z)$ for all $ z\in \mathbb{F}_{p^r} $, we have
$$
|G(A)|^2 = \sum_{x,y\in \F_{p^r}} \psi (A(x^2-y^2)).
$$
Since $p$ is odd, $2$ is invertible in $\F_{p^r}$.
Hence the change of variables $u=x-y$, $v=x+y$ is bijective, with inverse $x=(u+v)/2$, $y=(v-u)/2$. Therefore,
\begin{equation}\label{G(A)2}
    |G(A)|^2 = \sum_{u,v\in\F_{p^r}}\psi(Auv) = p^r + \sum_{\substack{u\in\F_{p^r}\\u\ne0}} \sum_{v\in\F_{p^r}}\psi(Auv).
\end{equation}
For each $u\ne0$, $v\mapsto Au \cdot v$ is a bijection, so it follows that 
\begin{equation}\label{sum-zero}
    \sum_{v\in \F_{p^r}}\psi (Auv) = \sum_{w\in \F_{p^r}} \psi (w) = 0,
\end{equation} 
where the sum $ \sum_{w\in \F_{p^r}} \psi (w) $ vanishes since $\psi$ is nontrivial. 
Hence, combining \eqref{G(A)2} and \eqref{sum-zero}, we have $|G(A)| = p^{r/2}$. This shows the sharpness of the Weil bound. This sharpness is the reason why we obtain the near optimal Fourier decay $\widehat{\mu}(\xi) = O(|\xi|^{-\alpha})$ for all $\alpha<\dim_H(E)/2$. If the factor $p^{r/2}$ in the bound were replaced by $p^{\theta r}$ for some $\theta>1/2$ and we follow along the line of the proof, we will only obtain $\alpha<(1-\theta)\dim_H(E)$ and $\dim_F(E)\ge 2(1-\theta)\dim_H(E)$. This estimate would not suffice to show that \(E\) is a Salem set. Here, \(E\) and \(\mu\) denote the set and measure to be constructed in the coming sections.}
\end{rem}


\section{The key Weil estimate on $\R^1$}\label{Sec3}
Using the notation in the previous section, let $1\le r<q, D=\lceil q/r\rceil$, 
and $p>D$ be an odd prime. Fix a root \(\alpha\in\F_{p^r}\) of a monic irreducible polynomial of degree \(r\) over \(\F_p\). By \eqref{eq:digit-expansion}, every \(x\in\F_{p^r}\) can be written as
\begin{equation}\label{eq-x-expand}
x= \sum_{j=0}^{r-1} c_j(x)\cdot \alpha^j.
\end{equation}
Note that $c_j$ are $\F_p$-linear functionals. Following the previous section, we define the map $\Psi_{r,q}$ as in (\ref{eq:Psi}) and define $g_i(x)$, $i=0,1,\cdots,q-1$, to be the coordinates of $\Psi_{r,q}$, i.e.,
$$
\Psi_{r,q}(x) = (g_0(x),\cdots, g_{q-1}(x))
$$
where $g_i: \F_{p^r}\to \{0,1,\cdots, p-1\}$. Note that $g_j(x) = c_j(x)$ if $0\le j\le r-1$. Define the maps $b: \F_p^q\to \N$ and $ \mathsf{b}: \mathbb{F}_{p^r} \to \N $ by
\begin{equation}\label{eq:bdef_0}
b(y_0,\cdots,y_{q-1})=\sum_{j=0}^{q-1} y_j \cdot p^{j},
\end{equation}
\begin{equation}\label{eq:bdef}
\mathsf{b}(x) = b(\Psi_{r,q}(x))=\sum_{j=0}^{q-1} g_j(x) \cdot p^{j}. 
\end{equation}
For an integer $M\ge1$, define the normalized $M$-{\bf Dirichlet mask} function by
\begin{equation}\label{eq:dirichlet}
{\mathsf D}_M(\xi)=\frac1M\sum_{j=0}^{M-1}e^{-2\pi ij\xi}
=e^{-\pi i(M-1)\xi}\frac{\sin(\pi M\xi)}{M\sin(\pi \xi)}
\quad\textup{when}~ \xi\notin\mathbb Z,
\end{equation}
and ${\mathsf D}_M(\xi)=1$ when $\xi\in\mathbb Z$.
The function ${\mathsf D}_M$ is the Fourier transform of the uniform probability measure on $\{0,1,\ldots,M-1\}\subset\mathbb R$. The main goal of this section is to prove the following theorem via the Weil bound.

\begin{theorem}\label{thm3.1-Weil}
Let
\[
\mathcal B
=
\left\{
\mathsf{b}(x):x\in\mathbb F_{p^r}
\right\},
\]
where the digit map $\mathsf{b}(x)$ is defined in \eqref{eq:bdef}. Then
\[
\sup_{\xi\in\mathbb R}
\left|
\mathsf M_{\mathcal B}(\xi)-{\mathsf D}_{p^q}(\xi)
\right|
\le
(D-1)p^{-r/2}
\bigl(2\log(2p)\bigr)^q.
\]
\end{theorem}

The rest of this section will be devoted to proving the deterministic construction in Theorem \ref{thm3.1-Weil}.  We first begin with two elementary lemmas.
\begin{lemma}
\label{lem:geometric-sum}\label{sum_esti1}
Let $L$ be a positive integer. For every $\phi\in\mathbb{R}$,
\[
\left|\sum_{x=0}^{L-1} e^{-2\pi i\phi x}\right|
\le
\min\left\{
L,\frac{1}{2\|\phi\|}
\right\},
\]
where
\[
\|\phi\|
:=\min_{n\in\mathbb{Z}}|\phi-n|
\]
denotes the distance from $\phi$ to the nearest integer, and $1/0$ is interpreted as $+\infty$.
\end{lemma}

\begin{proof}
By the triangle inequality,
\[
\left|\sum_{x=0}^{L-1} e^{-2\pi i\phi x}\right|
\le \sum_{x=0}^{L-1}\left|e^{-2\pi i\phi x}\right|
=L.
\]
If $\Vert\phi\Vert=0$, then $\phi\in\mathbb{Z}$ and the sum equals $L$, so the claimed bound holds. Suppose now that $\Vert\phi\Vert>0$. The geometric series formula gives
\[
\left|\sum_{x=0}^{L-1} e^{-2\pi i\phi x}\right|
=
\left|\frac{1-e^{-2\pi i L\phi}}{1-e^{-2\pi i\phi}}\right|
=
\frac{|\sin(\pi L\phi)|}{|\sin(\pi\phi)|}
\le \frac{1}{\sin(\pi\Vert\phi\Vert)}\leq \frac{1}{2\|\phi\|},
\]
since  $\sin x\ge (\frac{2}{\pi})x$ for all $x\in [0,\pi/2]$. Combining this with the bound by $L$ proves the lemma.
\end{proof}

\begin{lemma}\label{lem:coeffsum}
Let $L\ge 2$ be a positive integer, and let $\theta\in\mathbb{R}$. Then the function 
\[
f_{L,\theta}(u)=e^{-2\pi i\theta u}, 
\qquad u\in\{0,1,\cdots, L-1\},
\]
satisfies
\[
\sum_{a=0}^{L-1}|\widehat{f}_{L,\theta}(a)| \le 2\log(2L),
\]
where $\widehat{f_{L,\theta}}(a) =\frac1L\sum_{u=0}^{L-1} f_{L,\theta}(u)e^{-2\pi iau/L}.$
\end{lemma}

\begin{proof}
We have 
\[
\widehat{f_{L,\theta}}(a) =\frac{1}{L}\sum_{u=0}^{L-1} e^{-2\pi i(\theta+a/L)u}.
\]
Applying Lemma \ref{sum_esti1} with $\phi=\theta+a/L$ gives
\begin{equation}\label{eq:coeffpoint}
|\widehat{f_{L,\theta}}(a)| \le \min\left\{ 1,\frac{1}{2L\Vert\theta+a/L\Vert}
\right\}.
\end{equation}

Let $t_1,\ldots,t_L\in[-1/2,1/2)$ be the representatives of the points $\theta+a/L$ modulo $1$, for $a=0,\ldots,L-1$, ordered such that
\[
|t_1|\le |t_2|\le\cdots\le |t_L|.
\]
Any two distinct representatives are at least $1/L$ apart.
For each $2\le j\le L$, the $j$ points $t_1,\ldots,t_j$
lie in the interval $[-|t_j|,|t_j|]$. As each distinct $t_j$ is at least $1/L$ apart, the interval $ [-|t_j|,|t_j|] $ has length at least $(j-1)/L$, it follows that
\begin{equation}\label{2tj}
    2|t_j|\ge\frac{j-1}{L}.
\end{equation}

For $ 1 \le j \le L $, let $a_j\in\{0,1,\cdots, L-1\}$ be the unique element
such that $ \theta+a_j/L\equiv t_j (\textup{mod }1)$. Since $t_j\in[-1/2,1/2)$, we have
$\|\theta+a_j/L\|=|t_j|$.
Thus, by \eqref{eq:coeffpoint} and the inequality
\eqref{2tj}, we have
\[
|\widehat{f_{L,\theta}}(a_j)| \le \frac{1}{2L|t_j|} \le \frac{1}{j-1},
\]
for $ 2\le j\le L. $ Using this and bounding the first term $ |\widehat{f}_{L,\theta}(a_1)| $ by $ 1 $, we obtain
\[
\sum_{j=1}^{L}|\widehat{f_{L,\theta}}(a_j)|
\le 1+\sum_{j=2}^{L}\frac{1}{j-1}
\le 2+\log (L-1)\le 2\log (2L),
\]
where we used the fact that  $\sum_{k=1}^{L-1}1/k\le 1+\log (L-1)$. This completes the proof. 
\end{proof}

We also need a computational lemma about the Wiener norm which holds in every dimension and will be used in higher dimensions.

\begin{lemma}\label{lemma-product-form}
Let $m,d\ge1$ be integers, and let $\boldsymbol\omega_0,\ldots,\boldsymbol\omega_{m-1}\in\R^d$.
For $0\le j<m$, define $f_j:\Fp^d\to\C$ by
$ f_j(\mathbf x) = e^{-2\pi i\boldsymbol\omega_j\cdot\mathbf x}. $
Define $F:(\Fp^d)^m\to\C$ by
$$
F(\mathbf x_0,\ldots,\mathbf x_{m-1}) = \prod_{j=0}^{m-1}f_j(\mathbf x_j),
$$
where $ {\bf x}_j \in \Fp^d $ for $ 0 \le j < m $. 
Then, for
$\boldsymbol\eta=(\boldsymbol\eta_0,\ldots,\boldsymbol\eta_{m-1})\in(\Fp^d)^m$,
\[
\widehat F(\boldsymbol\eta) = \prod_{j=0}^{m-1}\widehat f_j(\boldsymbol\eta_j), \ \mbox{and} ~~
 \|F\|_W = \prod_{j=0}^{m-1}\|f_j\|_W \le \bigl(2\log(2p)\bigr)^{md}.
\]
\end{lemma}

\begin{proof}
 The Fourier transform formula follows from the fact that Fourier transform of the tensor product $ (f\otimes g) (x,y) := f(x)g(y)$ in $ L^2(G\times H)$ is equal to $\widehat{f}\otimes\widehat{g}$ for any locally compact abelian groups $G,H$ (See e.g. \cite[p.192, Exercise 1.12.14]{Tao2010Epsilon}). For completeness, we compute it directly. For
$\boldsymbol\eta=(\boldsymbol\eta_0,\ldots,\boldsymbol\eta_{m-1})\in(\Fp^d)^m$,
\begin{equation}\label{key-identity}
\begin{aligned}
\widehat F(\boldsymbol\eta)
&= \frac1{p^{md}}\sum_{\mathbf x_0,\ldots,\mathbf x_{m-1}\in\Fp^d}
\prod_{j=0}^{m-1} \left(f_j(\mathbf x_j) e^{-2\pi i\boldsymbol\eta_j\cdot\mathbf x_j/p} \right)\\
&= \prod_{j=0}^{m-1}\left(\frac1{p^d}\sum_{\mathbf x_j\in\Fp^d}
 f_j(\mathbf x_j)e^{-2\pi i\boldsymbol\eta_j\cdot\mathbf x_j/p}\right)\\
&=\prod_{j=0}^{m-1}\widehat f_j(\boldsymbol\eta_j).
\end{aligned}
\end{equation}
Consequently,
\begin{align*}
\|F\|_W
&=\sum_{\boldsymbol\eta_0,\ldots,\boldsymbol\eta_{m-1}\in\Fp^d}\prod_{j=0}^{m-1} |\widehat f_j(\boldsymbol\eta_j)|\\
&=\prod_{j=0}^{m-1}\left(\sum_{\boldsymbol\eta_j\in\Fp^d}|\widehat f_j(\boldsymbol\eta_j)|\right)\\
&=\prod_{j=0}^{m-1}\|f_j\|_W.
\end{align*}

For $ 0 \le j \le m - 1 $, to estimate $ \|f_j\|_W $, we write $\boldsymbol\omega_j=(\omega_{j,1},\ldots,\omega_{j,d})$ and define $h_{j,s}:\Fp\to\C$ by
$$
h_{j,s}(a)=e^{-2\pi i\omega_{j,s}a}, \qquad 1\le s\le d.
$$
Then
$$
 f_j(x_1,\ldots,x_d) = \prod_{s=1}^{d}h_{j,s}(x_s).
$$
For fixed $j$, apply \eqref{key-identity} with
the parameters $(m,d)$ replaced by $(d,1)$,
replacing $F$ by $f_j$ and $f_{s-1}$ by
$h_{j,s}$ for $1\le s\le d$.
It follows, for
$\boldsymbol\eta=(\eta_1,\ldots,\eta_d)\in\Fp^d$,
$$
 \widehat f_j(\boldsymbol\eta)
 =\prod_{s=1}^{d}\widehat h_{j,s}(\eta_s).
$$
Summing absolute values over $\boldsymbol\eta\in\Fp^d$
and separating the sums, we obtain
$$
 \|f_j\|_W
 =\prod_{s=1}^{d}
 \left(\sum_{b=0}^{p-1}|\widehat h_{j,s}(b)|\right)
 \le \bigl(2\log(2p)\bigr)^d,
$$
where the inequality follows from
Lemma~\ref{lem:coeffsum} with $L=p$ and
$\theta=\omega_{j,s}$ for each $s$.
Multiplying these estimates over $j=0,\ldots,m-1$ finishes the proof.
\end{proof}

\begin{proof}[Proof of Theorem \ref{thm3.1-Weil}]
The first $r$ entries of $\Psi_{r,q}(x)$ equal $C(x)$, so the map ${\mathsf b}(x)$ is injective and its range has $p^r$ elements. For fixed $\xi\in\R$, define $F_{\xi}: \Fp^q\to \C$ by
\[
 F_\xi(y_0,\ldots,y_{q-1})= \exp\left(-2\pi i\xi\sum_{j=0}^{q-1}p^jy_j\right),
 \qquad y_j\in\{0,\ldots,p-1\}.
\]
Then by the definition of $\mathsf{b}(x)$ in (\ref{eq:bdef}), $F_\xi(\Psi_{r,q}(x))=e^{-2\pi i\xi \mathsf{b}(x)}$. We now compute every term in (\ref{eq:averaging}). The first average in \eqref{eq:averaging} is
\begin{equation}\label{eq-compute1}
 p^{-r}\sum_{x\in\mathbb F_{p^r}}F_{\xi}(\Psi_{r,q}(x)) =\mathsf{M}_{\B}(\xi).\end{equation}
By the uniqueness of base-$p$ expansions, the map $(y_0,\ldots,y_{q-1})\mapsto\sum_{j=0}^{q-1}p^jy_j$ bijects $\{0,\ldots,p-1\}^q$ onto $\{0,\ldots,p^q-1\}$. Thus the second average sum in \eqref{eq:averaging} equals
\begin{equation}\label{eq-compute2}
 p^{-q}\sum_{y\in\Fp^q}F_\xi (y)
 =p^{-q}\sum_{k=0}^{p^q-1}e^{-2\pi i\xi k}={\mathsf{D}}_{p^q}(\xi).
\end{equation}
Finally, we note that $F_{\xi}(y_0,\cdots, y_{q-1}) = \prod_{j=0}^{q-1}e^{-2\pi i \xi p^j y_j}$, which is of the form given in Lemma \ref{lemma-product-form} (with $d=1$ and $m=q$). Hence, we obtain 
\begin{equation}\label{eq-compute3}
\Vert F_{\xi} \Vert_W = \sum_{\mathbf a\in\Fp^q}|\widehat{F}_\xi(\mathbf a)| \le\big(2\log(2p)\big)^q.
\end{equation}
Substituting (\ref{eq-compute1}), (\ref{eq-compute2}) and (\ref{eq-compute3}) into the estimate \eqref{eq:averaging} in Proposition \ref{thm-weil} completes the proof of this theorem.
\end{proof}

Let's end this section with two remarks. 

\begin{rem}
    We notice that if we have a Salem set $S\subset \F_p^q$ in the sense of \cite{IosevichRudnev2007}, we can also define 
    $$
    \B_S = \left\{\sum_{j=0}^{q-1}p^j y_j: (y_0,\cdots,y_{q-1})\in S\right\}
    $$
    Applying the same proof of Theorem \ref{thm3.1-Weil}, we can also obtain the same estimate for $\B_S$. We may regard our Moran construction of Salem sets as a recursive iteration of certain appropriately chosen discrete Salem sets. 
\end{rem}

\begin{rem}\label{rem3.6}We remark that, as noted in the introduction, a probabilistic analogue of the
mask approximation estimate in Theorem~\ref{thm3.1-Weil}
follows from \cite[Lemma 6.2]{LabaPramanik2009}. Indeed, the following holds: {there exists a constant $C>0$ such that,} for every integer $p\ge2$ and integers $ r, q $ with $1\le r<q$, there exists a set $ \mathcal{B} \subseteq\{0,1,\ldots,p^q-1\}$ with $ | \mathcal{B} |=p^r, $ satisfying
\[
\sup_{\xi\in\mathbb R}
\left|{\mathsf M}_{\mathcal{B}}(\xi)-\mathsf D_{p^q}(\xi)\right|
\le C\,p^{-r/2}\sqrt{q\log(2p)}.
\]
From \cite[Lemma 6.2]{LabaPramanik2009}, by taking $M=N=p^q$ and $t = p^r$ in their lemma,  it can be shown for all $ k \in \mathbb{Z} $,
\begin{equation}\label{mean-sup}
    \sup_{k\in \mathbb{Z}} \left\vert \mathsf{M}_{\mathcal{B}}\left( \dfrac{k}{N^2} \right) - {\mathsf D}_{N}\left( \dfrac{k}{N^2} \right)  \right\vert \le \sqrt{ \dfrac{32\log(8N^3)}{t} }. 
\end{equation}
For any $ \xi \in \mathbb{R} $, define $ F(\xi) = \mathsf{M}_{\mathcal{B}}(\xi) - \mathsf{D}_N(\xi) $. Direct calculation gives that  
\begin{equation}\label{derivative-upp}
    \vert F'(\xi) \vert \le 4\pi( N-1 ),
\end{equation} 
for all $ \xi \in \mathbb{R} $. For every $ \xi \in \mathbb{R} $, there exists $ k \in \mathbb{Z} $ such that $ \left\vert \xi - \dfrac{k}{N^2} \right\vert \le \dfrac{1}{2N^2}. $
By \eqref{mean-sup} and \eqref{derivative-upp}, we obtain
\begin{align*}
|F(\xi)|
&\le \left|F\left(\frac{k}{N^2}\right)\right|
+\sup_{u\in\mathbb R}|F'(u)|
 \left|\xi-\frac{k}{N^2}\right|\\
&\le \sqrt{\frac{32\log(8N^3)}{t}}+\frac{2\pi}{N}\\
&= p^{-r/2}\sqrt{32\log 8+96q(\log p)}
+2\pi p^{-q}\\
&\le  C p^{-r/2}\sqrt{q\log(2p)},
\end{align*}
for some $C>0$ (since $r<q$, the first term dominates).
This gives the {\it probabilistic} version of Theorem \ref{thm3.1-Weil}. In fact, interested readers can replicate the argument in Sections \ref{Sec4} and \ref{Sec5} to construct a probabilistic Moran Salem measure. However, spectrality cannot be guaranteed in the probabilistic construction. 
\end{rem}

\section{Constructions of the target Moran sets in $\R^1$}\label{Sec4}
In this section, we will define all the parameters required for the construction of the Moran sets in $ \mathbb{R}^1 $. We prove that these sets are Salem in Section~\ref{Sec5} and that their associated Cantor--Moran
measures are spectral in Section~\ref{Sec9}.

Fix $ 0 < s \le 1 $. The parameters below depend on $s$.
For simplicity, we suppress this dependence in the notation
throughout the remainder of the paper. If $ 0 < s < 1 $, set 
\begin{equation}\label{eq:qrdef}
    q_n = n + \left\lceil\frac2{s}\right\rceil, \quad r_n = \left\lfloor sq_n \right\rfloor. 
\end{equation}
If $ s = 1 $, set $ q_n = n + 2, r_n = n + 1. $ It follows 
\begin{equation}\label{properties-rnqn}
    2 \le r_n < q_n, \quad q_n \to \infty, \quad \frac{r_n}{q_n} \to s.
\end{equation}
The exact expression of $r_n,q_n$ is not important. The subsequent argument uses the properties in \eqref{properties-rnqn}.  With the above choices, one has $ r_n \ge 2 $ for every $ n $. 
Define
\begin{equation}\label{eq:epsbeta}
D_n=\left\lceil\frac{q_n}{r_n}\right\rceil.
\end{equation}
Note that $D_n\ge 2$ for all $ n \in \mathbb{N} $ in this construction. Set $p_0=1$. Having chosen integers \(p_0,\ldots,p_{n-1}\), let $p_n$ be the least odd prime $p>D_n$ satisfying
\begin{equation}\label{eq:scale1}
p^{q_n}> (p_1^{q_1}\cdots p^{q_{n-1}}_{n-1})^{\,n}
\end{equation}
and
\begin{equation}\label{eq:scale2}
(D_n-1) \bigl(2\log(2p)\bigr)^{q_n} \le p^{1/2}.
\end{equation}

\begin{lemma}\label{lem:primeexist}
For every $n$, there are infinitely many odd primes satisfying \eqref{eq:scale1} and \eqref{eq:scale2}. Consequently, the construction above yields a well-defined sequence of primes $\{p_n\}_{n\ge 1}$.
\end{lemma}

\begin{proof}
For each fixed $n$, the quantities $ q_n, r_n $ and $ D_n $ are fixed. Condition \eqref{eq:scale1} holds for every sufficiently large $p$. Moreover, as $ p \to \infty $, 
$$ (D_n-1) \bigl(2\log(2p)\bigr)^{q_n}/ p^{1/2} \to 0. $$ 
Hence condition \eqref{eq:scale2} also holds for every sufficiently large $ p $. Therefore, every sufficiently large prime exceeds $D_n$, and satisfies both \eqref{eq:scale1} and \eqref{eq:scale2}. In particular, the least such prime $p_n$ exists at each step, so the construction of $ \{p_n\}_{n} $ is well defined.
\end{proof}

For each $n\ge1$, set
\begin{equation}\label{eq:NM}
N_n=p_n^{q_n},
\qquad
M_n=p_n^{r_n},
\qquad
P_n=N_1\cdots N_n,
\end{equation}
with the convention that $P_0=1$.
For each $n \ge 1$, let $\mathsf{b}_n$ be the digit map defined in
\eqref{eq:bdef} with
$
p=p_n, r=r_n$ and $ q=q_n.
$
Define
\begin{equation}\label{def_Bn}
    \mathcal B_n = \left\{
\mathsf{b}_n(x):x\in\F_{p_n^{r_n}}
\right\}.
\end{equation}
 For $ 0 < s \le 1 $, let 
\begin{equation}\label{eq:Ks}
K_s = \left\{\sum_{n=1}^\infty\frac{b_n}{P_n}:b_n\in \B_n\right\}
\end{equation}
be the homogeneous Moran set generated by $ \{(N_n,\mathcal{B}_n)\}_{n=1}^{\infty} $. We notice that the above construction is completely deterministic once $ s $ is given and the digit set $\B_n$ is based on a basis chosen from the finite field $\F_{p_n^{r_n}}$ over $\F_{p_n}$. 

The one-dimensional case of Theorem~\ref{thm:main} will follow if we can prove the following theorem for our constructed Moran sets and measures.

\begin{theorem}\label{main-theorem-detail}
    For each $0 < s\le 1$, let $K_s$ be the Moran set defined in (\ref{eq:Ks}). Let $\mu_s$ be the associated equal-weighted Cantor-Moran measure defined in (\ref{cantor-moran_measure}). Then 
    \begin{enumerate}
        \item For every $0<\alpha<s/2$, there exists a constant
$D_{s,\alpha} > 0 $ such that for any $ \xi \in \mathbb{R} $, 
\[
|\wh\mu_s(\xi)|
\le
D_{s,\alpha}(1+|\xi|)^{-\alpha}.
\]
\item $K_s$ is a Salem set of dimension $s$ and $\mu_s$ is a singular measure.
\item $\mu_s$ is a spectral measure with a spectrum 
\[
\Lambda=\bigcup_{n\ge1}\Lambda_n, \qquad \Lambda_n=L_1+P_1L_2+\cdots+P_{n-1}L_n,
\]
where 
\begin{equation}\label{def-Ln}
    L_n = p_n^{q_n-r_n}\left\{-\frac{p_n^{r_n}-1}{2},\cdots, -1,0,1, \cdots,\frac{p_n^{r_n}-1}{2}\right\}.
\end{equation}
    \end{enumerate}
\end{theorem}
We will prove parts~(1) and~(2) of Theorem~\ref{main-theorem-detail} in Section~\ref{proof_Salem}, and part~(3) in Section~\ref{Sec9}.

\section{Proof of the Salem property in $\R^1$}\label{proof_Salem}\label{Sec5}

For $ n \ge 1 $, let $ r_n $, $ q_n $ and $ D_n $ be chosen in \eqref{eq:qrdef} and \eqref{eq:epsbeta} respectively. Let $p_n$ be the least odd prime $p>D_n$ satisfying \eqref{eq:scale1} and \eqref{eq:scale2}. Let $ N_n $ and $ \mathcal{B}_n $ be chosen in \eqref{eq:NM} and \eqref{def_Bn} respectively. According to \eqref{eq-FT}, let $ \mathsf M_{\mathcal B_n}( \cdot) $ be the mask function associated with $ \mathcal{B}_n $.

Throughout this section, we fix $ 0 < s \le 1 $. Let $ K_s $ be the homogeneous Moran set defined in \eqref{eq:Ks}. Let $\mu_s$ be the associated equal-weighted Cantor-Moran measure defined in (\ref{cantor-moran_measure}). In this section, we will show that $ K_s $ is a Salem set. To begin, for $ n \ge 1 $, we define
\begin{equation}\label{etan-def}
    \eta_n : = \left\| \mathsf M_{\mathcal B_n}-{\mathsf D}_{N_n} \right\|_\infty. 
\end{equation}
The following lemma provides quantitative estimates for the decay of $\eta_n$ and the growth of $N_n$.

\begin{lemma}\label{main-esti}
\label{lem:scaleconseq}
For every $0 < s \le 1$, the following statements hold.
\begin{enumerate}
\item[(i)] For $ n \ge 1 $, $ \eta_n \le 1 $. In addition, $\eta_n\to0$.

\item[(ii)] Let $0< \gamma < s / 2$. Then there exists a constant
$A_{s,\gamma}\ge1$ such that for $ n \ge 1 $,
\begin{equation}\label{eq:etauniform}
\eta_n\le A_{s,\gamma}N_n^{-\gamma}.
\end{equation}
\item[(iii)]  \begin{equation}\label{eq:separation}
    N_n>P_{n-1}^{\,n}. 
\end{equation}
\end{enumerate}
\end{lemma}

\begin{proof}
Applying the Weil bound in Theorem~\ref{thm3.1-Weil} with
$
p=p_n,
r=r_n,
q=q_n$ and $
D=D_n,
$
then using our condition \eqref{eq:scale2} and recalling that $N_n = p_n^{q_n}$, we obtain
\begin{equation}\label{eq:eta}
\begin{aligned}
\eta_n
& \le (D_n-1) p_n^{-r_n/2} \bigl(2\log(2p_n)\bigr)^{q_n}\\
&\le p_n^{ -( r_n - 1 )/2 }\\
&= N_n^{- \frac{r_n - 1}{2q_n}}.
\end{aligned}
\end{equation}

For part~(i), since $ \frac{r_n - 1}{ 2q_n } > 0 $ for $ n \ge 1 $, by \eqref{eq:eta}, $ \eta_n \le 1 $. On the other hand, since $ \frac{r_n}{q_n} \to s$ and $  q_n \to \infty $,
we have $ \frac{r_n - 1}{2q_n} \to \frac{s}{2}  $. Moreover,
$N_n=p_n^{q_n}\ge3^{q_n}\to\infty$. It follows from
\eqref{eq:eta} that $\eta_n\to0$.

For part~(ii), fix $0<\gamma< s / 2$. Since $ \frac{r_n - 1}{2q_n} \to s/2$, choose
$n_\gamma\ge2$ such that $\frac{r_n - 1}{2q_n} \ge \gamma$ whenever
$n\ge n_\gamma$. Then for $n\ge n_\gamma,$ \eqref{eq:eta} gives
\[
\eta_n\le N_n^{-\frac{r_n - 1}{2q_n}}\le N_n^{-\gamma}.
\]
Define
$A_{s,\gamma}
=
\max\left\{
1,\max_{1\le n<n_\gamma}\eta_nN_n^{\gamma}
\right\}.$ Then \eqref{eq:etauniform} holds for every $n\ge1$.

For part~(iii), \eqref{eq:separation}  is just \eqref{eq:scale1} in the definition of $p_n$. 
\end{proof}

 This is a brief outline of the estimates. Part~(ii) shows that $\mathsf{M}_{\mathcal{B}_n}(\cdot)$ is well approximated by the $N_n$-Dirichlet mask ${\mathsf D}_{N_n}(\cdot)$. Consequently, an upper bound for ${\mathsf D}_{N_n}(\cdot)$ yields a corresponding upper bound for $\mathsf{M}_{\mathcal{B}_n}(\cdot)$, up to the approximation error. The error is around $N_n^{-\gamma}$ where $\gamma<s/2$. We will  provide a power decay bound of  $ \vert \widehat{\mu}_s(\xi) \vert $ when  $\xi\sim P_n$ in Proposition \ref{prop:blockwise} in terms of powers of $P_{n-1}$ and $N_n$.  Using Lemma \ref{main-esti} (iii), we can compare it with $\xi$ to obtain the desired bound. 
\begin{lemma}\label{lem:Dbounds}
Let $M,N\ge2$ be integers, and let $t\in\R$ satisfy $1\le t\le N/2$. If $\lVert t\rVert>0$, then
\begin{equation}\label{eq:Dbounds}
|{\mathsf D}_M(t)|\le\min\left\{1,\frac1{2M\lVert t\rVert}\right\},
\qquad
|{\mathsf D}_N(t/N)|\le\frac{\pi\lVert t\rVert}{2t}.
\end{equation}
If $\lVert t\rVert=0$, then ${\mathsf D}_N(t/N)=0$.
\end{lemma}

\begin{proof}
First, if $0 < \lVert t\rVert\le1/2$, it is well known
\begin{equation}\label{sin_ineq}
    |\sin(\pi t)|=\sin\bigl(\pi\lVert t\rVert\bigr)
\ge2\lVert t\rVert,
\qquad
|\sin(\pi t)|\le\pi\lVert t\rVert.
\end{equation}
Recall \eqref{eq:dirichlet}, since $ t \not\in \mathbb{Z} $, using the first inequality in \eqref{sin_ineq},
$$ |{\mathsf D}_M(t)| = \left\vert \frac{\sin(\pi M t)}{M\sin(\pi t)} \right\vert \le \dfrac{ 1 }{2M \Vert t \Vert}. $$ Together with \eqref{eq:dirichlet} 
and the fact that
$|{\mathsf D}_M(t)|\le1$, the first inequality in \eqref{eq:Dbounds} follows.

Moreover, since $0< t/N \le 1 /2$, by the first inequality of \eqref{sin_ineq}, we have
$ N\sin(\pi t/N)\ge2t. $ Combining this with the second inequality of \eqref{sin_ineq}, we obtain
\[
|{\mathsf D}_N(t/N)|
=\frac{|\sin(\pi t)|}{N\sin(\pi t/N)}
\le\frac{\pi\lVert t\rVert}{2t}.
\]

Finally, if $\lVert t\rVert=0$, then $t$ is an integer with
$1\le t<N$. Thus $\sin(\pi t)=0$ while $\sin(\pi t/N)\ne0$, and
therefore ${\mathsf D}_N(t/N)=0$.
\end{proof}

Now we are ready to give the following upper bound estimates of the Fourier transform $ \hat{\mu}_s(\cdot) $.
\begin{proposition}\label{prop:blockwise}

Fix $0<s\le1$ and $0<\gamma<s/2$. There exists a constant $C'_{s,\gamma}>0$ such that, for every $n\ge2$ and $\xi\in\R$ satisfying $P_{n-1}\le|\xi|<P_n$, we have
\begin{equation}\label{eq:scalar-bound}
 |\widehat\mu_s(\xi)| \le C'_{s,\gamma} \left( N_{n-1}^{-\gamma}t^{-1}+N_n^{-\gamma} \right),
\end{equation}
where $ t=\frac{|\xi|}{P_{n-1}}\in[1,N_n). $

\end{proposition}
\begin{proof}
Since $|\widehat\mu_s(-\xi)|=|\widehat\mu_s(\xi)|$, we can assume that $\xi\ge0$. Then $\xi=P_{n-1}t$. Write
\[
 \mathsf{A}_n(t)=\mathsf{M}_{\B_n}\left(\frac{t}{N_n}\right)       \mathsf{M}_{\B_{n+1}}\left(\frac{t}{N_nN_{n+1}}\right).
\]
By the definition of $\eta_n$ in \eqref{etan-def} and the triangle inequality, using that all mask functions have modulus at most $1$, we obtain
\begin{equation}\label{An(t)-diff}
\begin{aligned}
&\left| \mathsf{A}_n(t)- \mathsf{D}_{N_n}\left(\frac{t}{N_n}\right)       \mathsf{D}_{N_{n+1}}\left(\frac{t}{N_nN_{n+1}}\right)\right|\\
\le & \left|\mathsf{M}_{\B_n}\left(\frac{t}{N_n}\right)-\mathsf{D}_{N_n}\left(\frac{t}{N_n}\right)\right|     \left| \mathsf{M}_{\B_{n+1}}\left(\frac{t}{N_nN_{n+1}}\right) \right|\\
+ & \left|\mathsf{D}_{N_n}\left(\frac{t}{N_n}\right)\right|      \left|\mathsf{M}_{\B_{n+1}}\left(\frac{t}{N_nN_{n+1}}\right)                 -\mathsf{D}_{N_{n+1}}\left(\frac{t}{N_nN_{n+1}}\right)\right|\\
\le & \eta_n+\eta_{n+1}.
\end{aligned}
\end{equation}
Observe that the adjacent Dirichlet masks can be combined into one Dirichlet mask function as follows:
$$\mathsf{D}_{N_n}\left(\frac{t}{N_n}\right)            \mathsf{D}_{N_{n+1}}\left(\frac{t}{N_nN_{n+1}}\right)  = {\mathsf D}_{N_nN_{n+1}} \left(\frac{t}{N_nN_{n+1}}\right).
$$
Combining with \eqref{An(t)-diff}, we have 
\begin{equation}\label{eq:An}
 \left|\mathsf A_n(t)-\mathsf D_{N_nN_{n+1}}
 \left(\frac{t}{N_nN_{n+1}}\right)\right|
 \le\eta_n+\eta_{n+1}.
\end{equation}
Since every factor in the Fourier product \eqref{eq-FT} has modulus at most $1$, by retaining the factors at levels $n-1$, $n$, and $n+1$, we have
\begin{equation}\label{upp-bdd-mus}
\begin{aligned}
|\widehat\mu_s(\xi)|
&\le |\mathsf M_{\B_{n-1}}(t)|\,|\mathsf A_n(t)|\\
&\le|\mathsf M_{\B_{n-1}}(t)| \left|\mathsf D_{N_nN_{n+1}}\left(\frac{t}{N_nN_{n+1}}\right)\right| + (\eta_n+\eta_{n+1})|\mathsf M_{\B_{n-1}}(t)|\\
& \le \bigl(|\mathsf D_{N_{n-1}}(t)|+\eta_{n-1}\bigr) \left|\mathsf D_{N_nN_{n+1}}\left(\frac{t}{N_nN_{n+1}}\right)\right|+\eta_n+\eta_{n+1}.
\end{aligned}
\end{equation}
Since $N_{n+1}\ge2$, we have $1\le t<N_n\le N_nN_{n+1}/2$. For $t\notin\Z$, applying Lemma~\ref{lem:Dbounds} with $M=N_{n-1}$ and $N=N_nN_{n+1}$, we have
\begin{equation}\label{D_n-1,n,n+1}
    |\mathsf{D}_{N_{n-1}}(t)|\le \frac{1}{2N_{n-1}\|t\|},\qquad \left|\mathsf{D}_{N_nN_{n+1}}\!\left(\frac{t}{N_nN_{n+1}}\right)\right| \le\frac{\pi\|t\|}{2t}\le \frac{\pi}{4t}.
\end{equation} 
In particular, using the inequalities in \eqref{D_n-1,n,n+1},
\begin{equation}\label{prod-upp-1}
    |\mathsf D_{N_{n-1}}(t)| \left| \mathsf D_{N_nN_{n+1}} \left(\frac{t}{N_nN_{n+1}}\right) \right| \le \frac1{2N_{n-1}\|t\|} \frac{\pi\|t\|}{2t}
 = \frac{\pi}{4N_{n-1}t}.
\end{equation}
If $t\in\Z$, then $\mathsf D_{N_nN_{n+1}}(t/(N_nN_{n+1}))=0$ by Lemma~\ref{lem:Dbounds}. Therefore, the second inequality in \eqref{D_n-1,n,n+1} and the inequality \eqref{prod-upp-1} still hold. Substituting these bounds into \eqref{upp-bdd-mus}, we obtain, for every $1\le t<N_n$,
$$
 |\widehat\mu_s(\xi)|
 \le
 \frac{\pi}{4t}
 \left(N_{n-1}^{-1}+\eta_{n-1}\right)
 +\eta_n+\eta_{n+1}.
$$

We now use $\eta_n\le A_{s,\gamma}N_n^{-\gamma}$ by (\ref{eq:etauniform}) in Lemma \ref{lem:scaleconseq} (ii) to obtain
$$
\begin{aligned}
|\widehat\mu_s(\xi)|\le& \frac{\pi}{4t}(N_{n-1}^{-1}+A_{s,\gamma}N_{n-1}^{-\gamma})+A_{s,\gamma}N_n^{-\gamma}+A_{s,\gamma}N_{n+1}^{-\gamma}\\
\le & \frac{\pi}{4t}(1+A_{s,\gamma})N_{n-1}^{-\gamma}+ 2A_{s,\gamma} N_n^{-\gamma}
\end{aligned}
$$
where we used $N_{n+1}>N_n$ and $0<\gamma<1$ in the last inequality. Since $A_{s,\gamma}\ge1$, we have
$$
 \frac{\pi}{4}(1+A_{s,\gamma}) \le\frac{\pi}{2}A_{s,\gamma} \le2A_{s,\gamma}.
$$
Thus \eqref{eq:scalar-bound} holds with $C'_{s,\gamma}=2A_{s,\gamma}$.
\end{proof}

Fix $0<\alpha<\gamma$.  We will see that the two terms in Proposition \ref{prop:blockwise} can be bounded by $|\xi|^{-\alpha}$ up to some positive constant for all sufficiently large $n$, thereby proving  $\lvert \widehat{\mu}_s(\cdot)\rvert$ decays at a polynomial rate.

\medskip

\noindent{\bf Proof of Theorem \ref{main-theorem-detail} (1).}
Fix $0<\alpha<s/2$ and choose $\gamma$ such that
$\alpha< \gamma < s/2$. We can find $n_{0}\ge 3$ such that for all $n \ge n_{0}$, we have
\begin{equation}\label{eq:choice-of-large-n}
\frac{\gamma(n-1)}{n}\ge\alpha,
\qquad
n\left(\gamma-\alpha\right)\ge\alpha.
\end{equation}
Let $\xi\ge P_{n_0-1}$. Then there is a unique $n\ge n_0$ such that
$P_{n-1}\le\xi<P_n$.  Set
\[
t=\frac{\xi}{P_{n-1}}.
\]
Then $\xi=P_{n-1}t$ and $1\le t<N_n$. By \eqref{eq:separation} with $n-1$,
\begin{equation}\label{eq:Nn-separation}
N_{n-1}>P_{n-2}^{n-1}.
\end{equation}
Since $P_{n-1}=P_{n-2}N_{n-1}$, multiplying both sides of
\eqref{eq:Nn-separation} by $N_{n-1}^{n-1}$ yields
\[
N_{n-1}^n
>
P_{n-2}^{n-1}N_{n-1}^{n-1}
=P_{n-1}^{n-1}.
\]
Equivalently,
\begin{equation}\label{eq:Nnminusone-separation}
N_{n-1}>P_{n-1}^{(n-1)/n}.
\end{equation}
The first inequality of
\eqref{eq:choice-of-large-n} gives
\[
N_{n-1}^{-\gamma}t^{-1}\le P_{n-1}^{-\gamma(n-1)/n}t^{-1}
\le
P_{n-1}^{-\alpha}t^{-\alpha}
=
\xi^{-\alpha}.
\]
Also, $N_n>P_{n-1}^n$ gives 
\[
N_n^{-\gamma} = N_n^{-(\gamma-\alpha)} N_n^{-\alpha}< P_{n-1}^{-n(\gamma-\alpha)}N_n^{-\alpha} \le P_{n-1}^{-\alpha}N_n^{-\alpha}
=P_n^{-\alpha}
\le\xi^{-\alpha}.
\]
Using Proposition~\ref{prop:blockwise}, we get
\begin{equation}\label{upper_bbd_mus}
    |\wh\mu_s(\xi)| \le 2 C_{s,\gamma}' \vert \xi \vert^{-\alpha}
\end{equation}
for every $\xi\ge P_{n_{0}-1}$.

Note that $\wh\mu_s(-\xi)=\overline{\wh\mu_s(\xi)}$ for $ \xi \in \mathbb{R} $, so the same estimate \eqref{upper_bbd_mus} also holds
for negative $\xi \le -P_{n_{0}-1}$. Thus,
$$ |\wh\mu_s(\xi)| \le 2 C_{s, \gamma}'|\xi|^{-\alpha}, $$
for all $ |\xi|\ge P_{n_{0}-1}$. It is now a standard argument to obtain that for some constant $D_{s,\alpha}>0$, we have 
\[
|\wh\mu_s(\xi)|
\le
D_{s,\alpha}(1+|\xi|)^{-\alpha},
\]
for all $ \xi \in \mathbb{R} $.
\qquad$\Box$

\medskip

In the remainder of this section, we will show that $K_s$ is a Salem set of dimension $s$, thereby establishing part~(2) of Theorem~\ref{main-theorem-detail}.

\noindent{\bf Proof of Theorem \ref{main-theorem-detail} (2).}
Set $ \rho_j = r_j / q_j $ for $ j \ge 1 $. By \eqref{eq:NM} and \eqref{def_Bn}, recall that $|\mathcal B_j|=M_j=p_{j}^{r_j}$ for $ j \ge 1 $. By \eqref{eq:NM} and \eqref{properties-rnqn}, we have
$\log M_j=\rho_j\log N_j$ for all $ j \ge 1 $, and $\rho_j\to s$. It follows
\[
\frac{\log ( M_1\cdots M_n )}{\log ( N_1 \cdots N_n )}
=
\frac{\sum_{j=1}^n\rho_j\log N_j}
     {\sum_{j=1}^n\log N_j}
\longrightarrow s,
\]
by using the classical Stolz--Ces\`aro theorem\footnote{Stolz--Ces\`aro theorem: Let $(a_n)$ and $(b_n)$ be real sequences such that $b_n$ is strictly increasing and $b_n\to+\infty$. Then $ \lim_{n\to\infty} \frac{a_{n+1}-a_n}{b_{n+1}-b_n} = L$ where $L\in\mathbb{R}\cup\{\pm\infty\}$ implies that $ \lim_{n\to\infty}\frac{a_n}{b_n}=L.$ We choose $a_n,b_n$ to be the partial sums in the numerator and denominator, respectively.}.   The level-$n$ construction covers $K_s$ by $ M_1\cdots M_n $ intervals,
each of length $ P_n^{-1} = (N_1\cdots N_n)^{-1} $. Hence, for every $u>s$,
\begin{align*}
    \mathcal H^u_{P_n^{-1}}(K_s)
&  \le (M_1\cdots M_n) (N_1\cdots N_n)^{-u} \\
&  = \exp\left[
\left(\frac{\log (M_1\cdots M_n)}{\log (N_1\cdots N_n)}-u\right)\log ( N_1\cdots N_n )
\right]
\longrightarrow0.
\end{align*}

It follows that $\dim_{\mathrm{H}}K_s\le s$. Let $\mu_s$ be the associated equal-weighted Cantor-Moran measure supported on $ K_s $ defined in (\ref{cantor-moran_measure}).
By part~(1) of Theorem~\ref{main-theorem-detail}, for every
$0<\alpha<s/2$ there exists $D_{s,\alpha} > 0 $ such that for $ \xi\in\mathbb R $,
\[
|\widehat{\mu_s}(\xi)|
\le D_{s,\alpha}(1+|\xi|)^{-\alpha}.
\]
Letting $ \alpha $ tend to $ s/2 $ gives $\dim_{\mathrm{F}}\mu_s\ge s$.
Since $\mu_s$ is supported on $K_s$ and Fourier dimension
does not exceed Hausdorff dimension (see e.g. \cite{falconer2004fractal}),
\[
s
\le\dim_{\mathrm{F}}\mu_s
\le\dim_{\mathrm{F}}K_s
\le\dim_{\mathrm{H}}K_s
\le s.
\]
It follows $K_s$ is a Salem set of dimension $s$. The measure is clearly singular when $s<1$.  When $s = 1$, we have $N_n = p_n^{n+2}$ and $M_n = p_n^{n+1}$, so the Lebesgue measure of $ K_s $ equals 
$$\lim_{n \to \infty} \frac{M_1\cdots M_n}{N_1\cdots N_n} = \lim_{n\to\infty}\frac{1}{p_1\cdots p_n} =0.$$ 
 Therefore, the measure $ \mu_s $ is still singular when $ s = 1 $.
\qquad$\Box$

\section{The key Weil estimate on $\R^d$}\label{Sec7}
The goal of this and the following two sections is to establish our main theorem in $\R^d$. We fix an integer $H\ge 1$ throughout and prove a Weil estimate with  the Dirichlet mask function raised to the power $H$.

 Fix $1\le r<q$ and an odd prime
\[
 p> D:=\left\lceil\frac{Hq}{r}\right\rceil.
\]
We are going to apply Proposition ~\ref{thm-weil} with  $ r\longmapsto dr$ and $q\longmapsto Hdq.$
Thus we take an irreducible polynomial of degree $dr$ over $\Fp$, write $\mathbb F_{p^{dr}}=\Fp(\alpha)$, and define
\begin{equation}\label{def-C(x)}
    x=\sum_{j=0}^{dr-1}c_j(x)\alpha^j, \qquad C(x)=(c_0(x),\ldots,c_{dr-1}(x)).
\end{equation}
The  power-coordinate map in (\ref{eq:Psi}) will be denoted by
$\Psi=\Psi_{dr,Hdq}: \mathbb F_{p^{dr}}\longrightarrow\mathbb F_p^{Hdq}$, which lists the first $Hdq$ coordinates of $x,x^2,\ldots,x^D$:
\begin{equation}\label{eq:Psi-high}
 \Psi(x)=\big(C(x),C(x^2),\ldots,C(x^D)\big)      
 \big|_{\text{first }Hdq\text{ entries}}\in\Fp^{Hdq}.
\end{equation}
Our goal is to produce a Weil estimate that can approximate the Dirichlet mask to the given power $H$. In the next subsection, we will first consider a more natural construction of such a digit set, but the digit may not be spectral in the group and it may have repeated elements. After that, we will provide the required spectral digit set and its Weil estimate. 

\subsection{The Weil estimate for a simpler digit set} We partition  $\Psi$ into two families of functions by grouping the coordinate functions of $\Psi$ into two families, as follows:
\begin{equation}\label{psi-e-f}
\Psi(x) = (e_0(x),\cdots, e_{Hdr-1}(x), f_0(x), \cdots, f_{Hd(q-r)-1}(x)) \in\Fp^{Hdq}. 
\end{equation}
Further partitioning $e_i$ and $f_j$ into $Hr$ and $H(q-r)$ many $d$-dimensional vectors, we define 
$$
 {\displaystyle \eb_j(x)=
 \begin{pmatrix}
 e_{jd}(x)\\e_{jd+1}(x)\\\vdots\\e_{(j+1)d-1}(x)
 \end{pmatrix}, 0\leq j<Hr,\qquad
 \fb_j(x)=
 \begin{pmatrix}
 f_{jd}(x)\\f_{jd+1}(x)\\\vdots\\f_{(j+1)d-1}(x)
 \end{pmatrix}, 0\leq j<H(q-r).}
$$
For $ k = 0, 1, \cdots, H - 1 $, define
\begin{equation}\label{def-uk-vk}
    {\bf u}_k(x) = \sum_{t=0}^{r-1}p^t \eb_{kr+t}(x), \quad {\bf v}_k(x) = \sum_{t=0}^{q-r-1}p^t \fb_{k(q-r)+t}(x). 
\end{equation}
For  $ {\bf y} \in (\Fp^d)^{Hq}$, write $ \mathbf y=(\mathbf y_0,\ldots,\mathbf y_{Hq-1}),$ where each $ \mathbf y_j\in\Fp^d. $ Define $ b_0:\Fp^{Hdq} \to \{0,\ldots,H(p^q-1)\}^d$ by
\begin{equation}\label{b_0(y)}
b_0(\mathbf y) = \sum_{k=0}^{H-1}\sum_{j=0}^{r-1} p^j\mathbf y_{kr+j} + p^r\sum_{k=0}^{H-1}\sum_{j=0}^{q-r-1} p^j\mathbf y_{Hr+k(q-r)+j}.
\end{equation}

Using the map $\Psi=\Psi_{dr,Hdq}$ defined in \eqref{eq:Psi-high}, define $ \bb_0:\mathbb F_{p^{dr}} \to \{0,\ldots,H(p^q-1)\}^d$ by 
$$\bb_0 = b_0 \circ\Psi. $$ 
By the definition of $ \Psi $ in \eqref{psi-e-f}.
$$
 \Psi(x)= \bigl( \eb_0(x),\ldots,\eb_{Hr-1}(x), \fb_0(x),\ldots,\fb_{H(q-r)-1}(x) \bigr).
$$
Consequently, by the definitions of $\mathbf u_k$ and $\mathbf v_k$,
\begin{equation}\label{eq-u_k-rep}
 \bb_0(x) = b_0(\Psi(x)) = \sum_{k=0}^{H-1}\mathbf u_k(x) + p^r\sum_{k=0}^{H-1}\mathbf v_k(x).
\end{equation}

Now we can define the digit set
$$
\B_0=\{\mathbf b_0(x):x\in\mathbb F_{p^{dr}}\}.
$$
{The map $\mathbf b_0$ need not be injective. Define
$$
 \mathsf M_{\B_0}(\xi) = p^{-dr}\sum_{x\in\mathbb F_{p^{dr}}} e^{-2\pi i\xi\cdot\mathbf b_0(x)}.
$$} The mask function $ \mathsf M_{\B_0}(\cdot) $ is also understood to be counting multiplicity of the elements. The following is our main result in this section. 
\begin{theorem}\label{thm:finite0}
Let $d,H\ge1$ be integers, and let $r,q$ be integers with $1\le r<q$. Let
$ D=\left\lceil\frac{Hq}{r}\right\rceil, $ and let $p>D$ be an odd prime. Then the mask function $ \mathsf{M}_{{\mathcal B}_0}(\cdot) $ satisfies
\begin{equation}\label{eq:maskerr0}
\sup_{\xi\in\R^d}\left|\mathsf M_{\B_0}(\xi)-\prod_{j=1}^d\bigl(\mathsf D_{p^q}(\xi_j)\bigr)^H\right| \le (D-1)p^{-dr/2} \bigl(2\log(2p)\bigr)^{Hdq},
\end{equation}
where $\mathsf D_{p^q}$ is the Dirichlet mask
defined in \eqref{eq:dirichlet}.
\end{theorem}

\begin{proof}
Recall that $ b_0(\Psi(x))=\bb_0(x) $ for $ x\in\mathbb F_{p^{dr}} $. For each $\xi\in\R^d$, define $F_{\xi, 0}:(\Fp^d)^{Hq}\to\C$ by
\begin{equation}\label{eq-F_xiy0}
 F_{\xi,0}(\mathbf y) = \exp\left(-2\pi i\xi\cdot b_0(\mathbf y)\right).
 \end{equation}
  We aim to apply Proposition \ref{thm-weil} (2) to $F_{\xi,0}$. Therefore, we need to compute three terms:
  $$
  p^{-dr}\sum_{x\in \F_{p^{dr}}} F_{\xi,0}(\Psi(x)), \qquad p^{-Hdq}\sum_{{\bf y}\in \F_p^{Hdq}}F_{\xi,0}({\bf y}) \qquad \mbox{and}  \qquad \|F_{\xi,0}\|_W. 
  $$
  For the first term, since $
F_{\xi,0}(\Psi(x)) = \exp\left(-2\pi i\xi\cdot  \bb_0(x)\right),$ it follows that 
\begin{equation}\label{eq-term1}
p^{-dr} \sum_{x\in\F_{p^{dr}}} F_{\xi,0} (\Psi(x))= p^{-dr} \sum_{x\in\F_{p^{dr}}} \exp\left(-2\pi i \xi\cdot{\mathbf{b}_0(x)}\right) = {\mathsf M}_{\B_0}(\xi). 
\end{equation}
For the second term, for $0\le k<H$, define
\begin{equation}\label{def-zk(y)}
    \mathbf z_k(\mathbf y) = \sum_{j=0}^{r-1}p^j\mathbf y_{kr+j} + \sum_{j=0}^{q-r-1} p^{r+j}\mathbf y_{Hr+k(q-r)+j}.
\end{equation}

By the definition of $F_{\xi,0}$,
$$
 F_{\xi,0}(\mathbf y) = \exp\left( -2\pi i\xi\cdot \sum_{k=0}^{H-1}\mathbf z_k(\mathbf y) \right).
$$
For $ k = 0,1, \cdots, H-1 $ and $ {\bf y} = ({\bf y}_0, \cdots, {\bf y}_{Hq - 1}) \in (\mathbb{F}_p^{d})^{Hq} $, by the definition of $ {\bf z}_k({\bf y}) $ in \eqref{def-zk(y)} and the uniqueness of base-$p$ expansions,
$$ ({\bf y}_{kr}, \cdots, {\bf y}_{kr + r-1}, {\bf y}_{Hr+k(q-r)}, \cdots,{\bf y}_{Hr+k(q-r)+q-r-1}) \longmapsto {\bf z}_k({\bf y})  $$
is a bijection between $\F_p^{dq}$ and $\{0,\ldots,p^q-1\}^d$. These  maps partition 
$({\bf y}_0,\cdots, {\bf y}_{Hq-1})$ into $H$ many disjoint groups of coordinates.  Combining these bijections, we conclude that 
$$
 \mathbf y\longmapsto \bigl(\mathbf z_0(\mathbf y),\ldots, \mathbf z_{H-1}(\mathbf y)\bigr)
$$
is a bijection from $\Fp^{Hdq}$ onto $\bigl(\{0,\ldots,p^q-1\}^d\bigr)^H$. 

Changing variables by this bijection and separating the finite sums,
$$
 p^{-Hdq}\sum_{\mathbf y\in\Fp^{Hdq}} F_{\xi,0}(\mathbf y) =  p^{-Hdq} \sum_{\mathbf z_0,\ldots,\mathbf z_{H-1} \in\{0,\ldots,p^q-1\}^d} e^{-2\pi i\xi\cdot\sum_{k=0}^{H-1}\mathbf z_k}.
$$
Factorizing the sum, we further have 
$$
\begin{aligned}
 p^{-Hdq} \sum_{\mathbf z_0,\ldots,\mathbf z_{H-1} \in\{0,\ldots,p^q-1\}^d} e^{-2\pi i\xi\cdot\sum_{k=0}^{H-1}\mathbf z_k} = &  \left( p^{-dq} \sum_{\mathbf z\in\{0,\ldots,p^q-1\}^d} e^{-2\pi i\xi\cdot\mathbf z} \right)^H\\
 = & \left( \prod_{i=1}^{d} \frac1{p^q}\sum_{a=0}^{p^q-1} e^{-2\pi i\xi_i a}
 \right)^H\\
 = & \prod_{i=1}^{d} \bigl(\mathsf D_{p^q}(\xi_i)\bigr)^H,
\end{aligned}
$$
where the last equality follows from \eqref{eq:dirichlet}. This gives
\begin{equation}\label{eq-term2}
 p^{-Hdq}\sum_{\mathbf y\in\Fp^{Hdq}} F_{\xi,0}(\mathbf y) = \prod_{i=1}^{d} \bigl(\mathsf D_{p^q}(\xi_i)\bigr)^H.
\end{equation}
Finally, by the definitions of $F_{\xi,0}(\cdot)$ and $ b_0(\cdot) $, for $ {\bf y} \in \mathbb{F}_p^{Hdq} $,
$$
\begin{aligned}
 F_{\xi,0}(\mathbf y) & = \exp\left(-2\pi i\xi\cdot b_0(\mathbf y)\right)  \\
 & = \prod_{k=0}^{H-1}\prod_{j=0}^{r-1} e^{-2\pi i p^j\xi\cdot\mathbf y_{kr+j}} \times \prod_{k=0}^{H-1}\prod_{j=0}^{q-r-1} e^{-2\pi i p^{r+j}\xi\cdot
 \mathbf y_{Hr+k(q-r)+j}}.
\end{aligned}
$$
For $\xi\in\R^d$, define $f_\ell:\Fp^d\to\C$, $0\le\ell<Hq$, by
$$
\begin{aligned}
 f_{kr+j}(\mathbf x) &=e^{-2\pi i p^j\xi\cdot\mathbf x}, &&0\le k<H,\quad 0\le j<r,\\
 f_{Hr+k(q-r)+j}(\mathbf x) &=e^{-2\pi i p^{r+j}\xi\cdot\mathbf x}, &&0\le k<H,\quad 0\le j<q-r.
\end{aligned}
$$
It follows
$$
F_{\xi,0}(\mathbf y) = \prod_{\ell=0}^{Hq-1}f_\ell(\mathbf y_\ell).
$$
Applying Lemma~\ref{lemma-product-form} with $m=Hq$, $F=F_{\xi,0}$, and
$\mathbf x_\ell=\mathbf y_\ell$ for $ 0\le\ell<Hq $, we obtain
\begin{equation}\label{eq-term3}
 \|F_{\xi,0}\|_W \le \bigl(2\log(2p)\bigr)^{Hdq},
\end{equation}
for every $\xi\in\R^d.$ Putting (\ref{eq-term1}), (\ref{eq-term2}) and (\ref{eq-term3}) into Proposition \ref{thm-weil} (2) completes the proof. 
\end{proof}

\subsection{The Weil estimate for the spectral digit sets} The digit set $\B_0$ unfortunately may have repeated elements, as we are summing $H$ vectors ${\bf u}_k(x)\in \{0,1,\cdots, p^r-1\}^d$ (see (\ref{eq-u_k-rep})) in the same region.  {This construction does not guarantee spectrality. When $\B_0$ has repeated elements, the pushforward measure
$$
 p^{-dr}\sum_{x\in\mathbb F_{p^{dr}}}\delta_{\mathbf b_0(x)}
$$
assigns mass $p^{-dr}\#\{x:\mathbf b_0(x)=\mathbf b\}$ to each $\mathbf b\in\B_0$. This measure need not be uniform.}  In this subsection, we are going to construct another digit set whose elements are distinct modulo the lattice $p^r\Z^d$. The idea is to perform a modulo correction. 

In the following, the notation $[z]_{p^r}$ denotes the unique integer in $\{0,\ldots,p^r-1\}$ congruent to $z$ modulo $p^r$, i.e. 
$$
z\equiv [z]_{p^r} \qquad (\mbox{mod} ~p^r).
$$
For $\mathbf{x}=(x_0,\ldots,x_{d-1})^{\top}\in\mathbb Z^d$, let $[\mathbf{x}]_{p^r}$ denote the vector obtained by replacing each coordinate $x_i$ with its least nonnegative residue modulo $p^r$. Thus,
$$
[\mathbf{x}]_{p^r}
=\bigl([x_0]_{p^r},\ldots,[x_{d-1}]_{p^r}\bigr)^{\top}.
$$ 
Define $b_s:\mathbb{F}_p^{Hdq}\to\{0,\ldots,H(p^q-1)\}^d$ by
\begin{equation}\label{b_s(y)}
b_s({\bf y}) = \left[\sum_{t=0}^{r-1}p^t\mathbf{y}_t-\sum_{k=1}^{H-1}\sum_{t=0}^{r-1}p^t\mathbf{y}_{kr+t}\right]_{p^r}+\sum_{k=1}^{H-1}\sum_{t=0}^{r-1} p^t\mathbf{y}_{kr+t}+p^r\sum_{k=0}^{H-1}\sum_{t=0}^{q-r-1} p^t\mathbf{y}_{Hr+k(q-r)+t}
\end{equation}
and
\begin{equation}\label{def-bs-map}
 \bb_s(x)=b_s(\Psi(x)), \qquad x\in\mathbb F_{p^{dr}},
\end{equation}
where $\Psi = \Psi_{dr,Hdq}$ is defined in (\ref{eq:Psi-high}). For $H\ge 2$, using the notation ${\bf u}_k(\cdot),{\bf v}_k(\cdot)$ in (\ref{def-uk-vk}) in the previous subsection, define
$$
    \bb_s(x) = b_s(\Psi(x)) = \left[ {\bf u}_0(x) - \sum_{k=1}^{H-1}\mathbf u_k(x) \right]_{p^r} + \sum_{k=1}^{H-1}\mathbf u_k(x) + p^r\sum_{k=0}^{H-1}\mathbf v_k(x).
$$
{\color{black}If $H=1$, we set $\bb_s=\bb_0$, where $\bb_0$ is defined in the previous subsection. Define the corresponding digit set by
\begin{equation}\label{def-bbs}
\B_s=\{\mathbf b_s(x):x\in\mathbb F_{p^{dr}}\}.
\end{equation}
The subscript $s$ stands for ``spectral''. The following lemma shows that the map $\bb_s$ has the desired properties for spectral sets.}

\begin{lemma}\label{lem:basic-property-highD}
The map $\bb_s$ defined in \eqref{def-bs-map} satisfies, for every $x\in\F_{p^{dr}}$,
\[
 \bb_s(x)\equiv{\bf u}_0(x)\pmod{p^r\Z^d},\qquad \bb_s(x)\in\{0,\ldots,H(p^q-1)\}^d.
\]
Moreover, the map $x\in\F_{p^{dr}} \longmapsto[\bb_s(x)]_{p^r}$
is bijective. Consequently, the set $ \B_s=\{\bb_s(x):x\in\F_{p^{dr}}\} $ contains exactly one representative of each residue class modulo $p^r\Z^d$. In particular, $|\B_s|=p^{dr}$.
\end{lemma}
\begin{proof}
By the definition of $\bb_s$, we have
\begin{equation}\label{bs(x)-congruence}
    \bb_s(x) \equiv \mathbf u_0(x)-\sum_{k=1}^{H-1}\mathbf u_k(x) +\sum_{k=1}^{H-1}\mathbf u_k(x) = \mathbf u_0(x) \pmod{p^r\Z^d}.
\end{equation}
By the definitions of $ {\bf u}_k $ and $ {\bf v}_k $ in \eqref{def-uk-vk}, for $ k = 0, 1, \cdots, H - 1 $, 
$$ \Vert {\bf u}_k \Vert_{\infty} \le p^r - 1, \quad \Vert {\bf v}_k \Vert_{\infty} \le p^{q-r} - 1. $$
Therefore,
$$
\begin{aligned}
    \Vert { \bf b }_s \Vert_{\infty} & \le  (p^r-1)+(H-1)(p^r-1)+p^rH(p^{q-r}-1)\\
 &=H(p^q-1),
\end{aligned}
$$
showing $ \bb_s(x)\in\{0,\ldots,H(p^q-1)\}^d.$ 

It remains to prove the bijectivity. By \eqref{eq:Psi-high} and \eqref{psi-e-f}, using \eqref{def-C(x)}, we have
$$
 e_i(x)=c_i(x),\qquad 0\le i<dr.
$$
Since $\{1,\alpha,\ldots,\alpha^{dr-1}\}$ is a basis of $\F_{p^{dr}}$ over $\Fp$, the map
$$
 x\longmapsto(c_0(x),\ldots,c_{dr-1}(x))
$$
is a bijection from $\F_{p^{dr}}$ onto $\Fp^{dr}$. For $1\le j\le d$ and $ x \in \mathbb{F}_{p^{dr}} $, by the definition of $\mathbf u_0$ in \eqref{def-uk-vk}, 
$$
 (\mathbf u_0(x))_j = \sum_{t=0}^{r-1}p^t e_{td+j-1}(x) = \sum_{t=0}^{r-1}p^t c_{td+j-1}(x).
$$
The indices $td+j-1$, with $0\le t<r$ and $1\le j\le d$, run through $0,\ldots,dr-1$ exactly once. By uniqueness of base-$p$ expansions, the map
$$
 x\longmapsto\mathbf u_0(x)
$$
is a bijection from $\F_{p^{dr}}$ onto $\{0,\ldots,p^r-1\}^d$.

By the congruence in \eqref{bs(x)-congruence}, we have $ [\bb_s(x)]_{p^r}=\mathbf u_0(x). $ Therefore, $x\mapsto[\bb_s(x)]_{p^r}$ is also a bijection onto $\{0,\ldots,p^r-1\}^d$.  Hence $\B_s$ contains exactly one representative of each class in $\Z^d/p^r\Z^d$, and $|\B_s|=p^{dr}$.
\end{proof}

The Weil estimate is given below:
\begin{theorem}\label{thm:finite}
Let $d,H\ge1$ be integers, and let $r,q$ be integers with $1\le r<q$. Let
$ D=\left\lceil\frac{Hq}{r}\right\rceil, $ and let $p>D$ be an odd prime. Then $\B_s$ satisfies
\begin{equation}\label{eq:maskerr}
\sup_{\xi\in\R^d}\left|\mathsf M_{\B_s}(\xi)-\prod_{j=1}^d\bigl(\mathsf D_{p^q}(\xi_j)\bigr)^H\right| \le (D-1)p^{-dr/2} \bigl(2\log(2p^r)\bigr)^d \bigl(2\log(2p)\bigr)^{Hdq},
\end{equation}
where $\mathsf D_{p^q}$ is the Dirichlet mask defined in \eqref{eq:dirichlet}.
\end{theorem}

Define $F_{\xi, s}:(\Fp^d)^{Hq}\to\C$ by
\begin{equation}\label{eq-F_xisy}
 F_{\xi,s}(\mathbf y) = \exp\left(-2\pi i\xi\cdot b_s(\mathbf y)\right).
\end{equation}
Then from $
F_{\xi,s}(\Psi(x)) = \exp\left(-2\pi i\xi\cdot  \bb_s(x)\right),$ it follows directly that
\begin{equation}\label{eq:F_xi_s}
p^{-dr} \sum_{x\in\F_{p^{dr}}} F_{\xi,s} (\Psi(x))={\mathsf M}_{\B_s}(\xi).
\end{equation}

Using the same argument as in Theorem \ref{thm:finite0} and Theorem \ref{thm3.1-Weil}, our proof will follow if we can compute the next two terms in Proposition \ref{thm-weil} (2). We recall them as below:

\begin{proposition}\label{prop-high-D-spectral-digit}
For $\xi\in\R^d$, the following statements hold.
\begin{enumerate}
\item \begin{equation}\label{eq-claim_highD}
\qquad p^{-Hdq} \sum_{{\bf y}\in \Fp^{Hdq}} F_{\xi,s}({\bf y}) = \prod_{j=1}^d({\mathsf D}_{p^q}(\xi_j))^H,
\end{equation}
\item  \begin{equation}\label{W-bound}
 {\displaystyle
\|F_{\xi,s}\|_W= \sum_{v\in\F_p^{Hdq}}|\widehat{F_{\xi,s}}(v)|\le \bigl(2\log(2p^r)\bigr)^d \bigl(2\log(2p)\bigr)^{Hdq}.}
\end{equation}

\end{enumerate}
\end{proposition}

\begin{proof}[Proof of Proposition \ref{prop-high-D-spectral-digit} (1)] Let $ {\bf y} = ( {\bf y}_0, \cdots, {\bf y}_{Hq-1}) \in \mathbb{F}_p^{Hdq} $. For $0\le k<H$, write
\begin{equation}\label{def-gk-hk}
 \mathbf g_k(\mathbf y) =\sum_{t=0}^{r-1}p^t\mathbf y_{kr+t}, \quad  \mathbf h_k(\mathbf y) =\sum_{t=0}^{q-r-1}p^t\mathbf y_{Hr + k(q-r) + t}.
\end{equation} 
Hence, recalling from (\ref{b_s(y)}) and (\ref{b_0(y)}) the definitions of $b_s(\bf y)$ and $b_0(\bf y)$ respectively, we have
\begin{equation}\label{bs(y)-gk(y)-hk(y)}
    b_s({\bf y}) = \left[{\bf g}_0({\bf y})-\sum_{k=1}^{H-1}{\bf g}_k(\bf y)\right]_{p^r} +\sum_{k=1}^{H-1}\mathbf g_k({\bf y}) +p^r\sum_{k=0}^{H-1}\mathbf h_k({\bf y}),
\end{equation}
$$
b_0({\bf y}) = \mathbf g_0({\bf y})+\sum_{k=1}^{H-1}\mathbf g_k({\bf y}) +p^r\sum_{k=0}^{H-1}\mathbf h_k({\bf y}).
$$
Fix $\mathbf y_r,\ldots,\mathbf y_{Hq-1}$. Write for simplicity that $\mathbf g_1 = \mathbf g_1({\bf y}),\ldots,\mathbf g_{H-1}=\mathbf g_{H-1}({\bf y})$ and $\mathbf h_0 = \mathbf h_0({\bf y}),\ldots,\mathbf h_{H-1} = \mathbf h_{H-1}(\bf y)$, and they are fixed since they do not depend on ${\bf y}_0,\cdots, {\bf y}_{r-1}$. By the uniqueness of base-$p$ expansions, as $\mathbf y_0,\ldots,\mathbf y_{r-1}$ vary over $ \mathbb{F}_p^d $, $\mathbf g_0({\bf y})$ runs exactly once through
$\{0,\ldots,p^r-1\}^d$. It follows
\begin{align*}
 &\sum_{\mathbf y_0,\ldots,\mathbf y_{r-1}\in\Fp^d} F_{\xi,s}(\mathbf y)\\
 =& \sum_{\mathbf g_0\in\{0,\ldots,p^r-1\}^d} \exp\left( -2\pi i\xi\cdot
 \left( \left[\mathbf g_0-\sum_{k=1}^{H-1}\mathbf g_k\right]_{p^r} +\sum_{k=1}^{H-1}\mathbf g_k +p^r\sum_{k=0}^{H-1}\mathbf h_k \right)
 \right)\\
 =& \sum_{\mathbf g_0\in\{0,\ldots,p^r-1\}^d} \exp\left( -2\pi i\xi\cdot
 \left( \mathbf g_0+\sum_{k=1}^{H-1}\mathbf g_k +p^r\sum_{k=0}^{H-1}\mathbf h_k
 \right) \right)\\
 =& \sum_{\mathbf y_0,\ldots,\mathbf y_{r-1}\in\Fp^d} F_{\xi,0}(\mathbf y).
\end{align*}

 For the second equality, since $\mathbf y_r,\ldots,\mathbf y_{Hq-1}$ are fixed, by the definition of $ \mathbf g_k(\cdot) $ in \eqref{def-gk-hk}, $\sum_{k=1}^{H-1}\mathbf g_k$ is fixed and dependent of $ \mathbf y_0,\ldots,\mathbf y_{r-1} $. Then the map
$$
 \mathbf g_0\longmapsto \left[\mathbf g_0-\sum_{k=1}^{H-1}\mathbf g_k\right]_{p^r}
$$
is a permutation of $\{0,\ldots,p^r-1\}^d$.  The last equality follows from the
definition of $F_{\xi,0}$.

Finally, we sum over the remaining coordinates $\mathbf y_r,\ldots,\mathbf y_{Hq-1}$ and divide by $p^{Hdq}$. 
\begin{align*}
 p^{-Hdq}\sum_{\mathbf y\in\Fp^{Hdq}} F_{\xi,s}(\mathbf y)
 &=
 p^{-Hdq}\sum_{\mathbf y\in\Fp^{Hdq}} F_{\xi,0}(\mathbf y)\\
 &= \prod_{j=1}^{d}\bigl(\mathsf D_{p^q}(\xi_j)\bigr)^H,
\end{align*}
where the last equality follows from
\eqref{eq-term2}. This proves \eqref{eq-claim_highD}.
\end{proof}
\begin{proof}[Proof of Proposition \ref{prop-high-D-spectral-digit} (2)] Fix $\xi\in\R^d$, and recall $\mathbf g_k,\mathbf h_k$
from \eqref{def-gk-hk}. For $ {\bf y} \in \mathbb{F}_{p}^{Hdq} $, by \eqref{bs(y)-gk(y)-hk(y)},
$$  
\begin{aligned}
F_{\xi,s}(\mathbf y) 
& = \exp\left(-2\pi i\xi\cdot b_s(\mathbf y)\right) \\
& = \exp\left( -2\pi i\xi\cdot \left( \left[\mathbf g_0-\sum_{k=1}^{H-1}\mathbf g_k\right]_{p^r} +\sum_{k=1}^{H-1}\mathbf g_k +p^r\sum_{k=0}^{H-1}\mathbf h_k \right) \right).
\end{aligned}
$$
For $ a \in \mathbb{Z} $, let $ f_{p^r,\xi_i}(a) = e^{-2\pi i \xi_i [a]_{p^r}} $. For $1\le i\le d$, applying Fourier inversion on the cyclic group $\Z/ p^r \mathbb{Z} $ to $f_{p^r,\xi_i}$, we have 
\begin{equation}\label{Fourier-inversion1}
    e^{-2\pi i\xi_i[a]_{p^r}} = \sum_{k_i=0}^{p^r-1} \widehat f_{p^r,\xi_i}(k_i) e^{2\pi i k_i a/p^r},
\end{equation}  
for $a\in\Z.$ For each $1 \le i \le d$, applying the identity \eqref{Fourier-inversion1} with
$$
 a= \left( \mathbf g_0(\mathbf y) -\sum_{k=1}^{H-1}\mathbf g_k(\mathbf y) \right)_i,
$$
which denotes the $ i $-th coordinate of $ \mathbf g_0(\mathbf y) -\sum_{k=1}^{H-1}\mathbf g_k(\mathbf y) $, and then multiplying the resulting identities over
$i=1,\ldots,d$, we obtain
\begin{equation}\label{inversion-g0-gk}
\begin{aligned}
 &e^{-2\pi i\xi\cdot [\mathbf g_0(\mathbf y) -\sum_{k=1}^{H-1}\mathbf g_k(\mathbf y)]_{p^r}}\\
 = & \prod_{i=1}^{d} \left( \sum_{k_i=0}^{p^r-1} \widehat f_{p^r,\xi_i}(k_i)
 e^{2\pi i(k_i/p^r) (\mathbf g_0(\mathbf y) -\sum_{k=1}^{H-1}\mathbf g_k(\mathbf y))_i} \right)\\
 = & \sum_{\kvec\in\{0,\ldots,p^r-1\}^d} \left( \prod_{i=1}^{d}\widehat f_{p^r,\xi_i}(k_i) \right) e^{2\pi i(\kvec/p^r)\cdot (\mathbf g_0(\mathbf y)
 -\sum_{k=1}^{H-1}\mathbf g_k(\mathbf y))},
\end{aligned}
\end{equation}
where $\kvec=(k_1,\ldots,k_d)$.  By the definition of $F_{\xi,s}$, multiplying the identity \eqref{inversion-g0-gk} by
$$
 e^{-2\pi i\xi\cdot\sum_{k=1}^{H-1}\mathbf g_k(\mathbf y)}  e^{-2\pi i p^r\xi\cdot\sum_{k=0}^{H-1}\mathbf h_k(\mathbf y)},
$$
we obtain
$$
 F_{\xi,s}(\mathbf y) = \sum_{\kvec\in\{0,\ldots,p^r-1\}^d} \left( \prod_{i=1}^{d}\widehat f_{p^r,\xi_i}(k_i) \right) e^{2\pi i(\kvec/p^r)\cdot\mathbf g_0(\mathbf y)}\times \prod_{k=1}^{H-1} e^{-2\pi i(\xi+\kvec/p^r)\cdot\mathbf g_k(\mathbf y)} \prod_{k=0}^{H-1} e^{-2\pi i p^r\xi\cdot\mathbf h_k(\mathbf y)}.
$$
For each $\kvec\in\{0,\ldots,p^r-1\}^d$, define $G_{\xi,\kvec}:(\Fp^d)^{Hq}\to\C$ by
\begin{equation}\label{def-G}
 G_{\xi,\kvec}(\mathbf y)= e^{2\pi i(\kvec/p^r)\cdot\mathbf g_0(\mathbf y)}\times \prod_{k=1}^{H-1} e^{-2\pi i(\xi+\kvec/p^r)\cdot\mathbf g_k(\mathbf y)} \prod_{k=0}^{H-1} e^{-2\pi i p^r\xi\cdot\mathbf h_k(\mathbf y)}.
\end{equation}
Therefore,
$$
 F_{\xi,s}(\mathbf y)=\sum_{\kvec\in\{0,\ldots,p^r-1\}^d} \left(\prod_{i=1}^d \widehat{f_{p^r,\xi_i}}(k_i)\right) G_{\xi,\kvec}(\mathbf y).
$$
Note that $\|\cdot\|_W$ is a norm on the functions of $(\F_p^{d})^{Hq}$. By the triangle inequality for the Wiener norm,
\begin{equation}\label{eq-triangle-wiener-new}
\|F_{\xi,s}\|_W  \leq \sum_{\kvec\in\{0,\ldots,p^r-1\}^d}
 \left(\prod_{i=1}^d \vert \widehat{f_{p^r,\xi_i}}(k_i) \vert \right)
 \|G_{\xi,\kvec}\|_W.
\end{equation}

We need to estimate the Wiener norm of $G_{\xi,\kvec}$. Using the definitions of $ {\bf g}_k $ and $ {\bf h}_k $ in \eqref{def-gk-hk}, we obtain
\begin{equation}\label{eq:Gfactor}
 G_{\xi,\kvec}(\mathbf y) = \prod_{t=0}^{r-1} e^{2\pi i p^t(\kvec/p^r)\cdot\mathbf y_t} \times \prod_{u=1}^{H-1}\prod_{t=0}^{r-1}
 e^{-2\pi i p^t(\xi+\kvec/p^r)\cdot \mathbf y_{ur+t}}\times
 \prod_{u=0}^{H-1}\prod_{t=0}^{q-r-1}e^{-2\pi i p^{r+t}\xi\cdot
 \mathbf y_{Hr+u(q-r)+t}}.
\end{equation}

For fixed $\kvec\in\{0,\ldots,p^r-1\}^d$, define $\boldsymbol\omega_0,\ldots, \boldsymbol\omega_{Hq-1}\in\R^d$ by
$$
\begin{aligned}
 \boldsymbol\omega_t
 &=-p^t\frac{\kvec}{p^r}, &&0\le t<r,\\
 \boldsymbol\omega_{ur+t} &=p^t\left(\xi+\frac{\kvec}{p^r}\right),
 &&1\le u<H,\quad 0\le t<r,\\
 \boldsymbol\omega_{Hr+u(q-r)+t} &=p^{r+t}\xi, &&0\le u<H,\quad 0\le t<q-r.
\end{aligned}
$$
Then \eqref{eq:Gfactor} can be rewritten as 
$$
G_{\xi,\kvec}(\mathbf y) = \prod_{t=0}^{Hq-1}  e^{-2\pi i\boldsymbol\omega_t\cdot\mathbf y_t}.
$$
Apply Lemma~\ref{lemma-product-form} with $m = Hq$ and $ f_j({\bf x}_j) = e^{-2\pi i\boldsymbol\omega_j\cdot\mathbf x_j}  $ for $ 0 \le j < Hq $. Then for every $\xi\in\R^d$ and $\kvec\in\{0,\ldots,p^r-1\}^d$, we have 
\begin{equation}\label{eq:Gbound-new}
 \|G_{\xi,\kvec}\|_W \le \bigl(2\log(2p)\bigr)^{Hdq}.
\end{equation}

Applying Lemma~\ref{lem:coeffsum} for $ L = p^r $ and $ \theta = \xi_i $ for $ i=1, \cdots, d $, we obtain 
\begin{equation}\label{eq:cyclic-bound-new}
{\displaystyle
 \sum_{k=0}^{p^r-1}| \widehat f_{p^r,\xi_i}(k)|\le2\log(2p^r),}
\end{equation}
for $ i=1, \cdots, d $.

Combining \eqref{eq-triangle-wiener-new} with the uniform bound \eqref{eq:Gbound-new}, and then using \eqref{eq:cyclic-bound-new}, we obtain
\begin{align*}
\|F_{\xi,s}\|_W
&\le\bigl(2\log(2p)\bigr)^{Hdq}\sum_{\mathbf k\in\{0,\ldots,p^r-1\}^d}
\prod_{i=1}^d |\widehat f_{p^r,\xi_i}(k_i)|\\
&=\bigl(2\log(2p)\bigr)^{Hdq}\prod_{i=1}^d\left( \sum_{k_i=0}^{p^r-1}
|\widehat f_{p^r,\xi_i}(k_i)|\right)\\
&\le \bigl(2\log(2p^r)\bigr)^d \bigl(2\log(2p)\bigr)^{Hdq}.
\end{align*}
This proves \eqref{W-bound}.
\end{proof}

\begin{proof}[Proof of Theorem~\ref{thm:finite}] We apply Proposition \ref{thm-weil} (2) to $F_{\xi,s}$,
with $(r,q)$ replaced by $(dr,Hdq)$. By  plugging in the result of Proposition \ref{prop-high-D-spectral-digit} and (\ref{eq:F_xi_s}) into  Proposition~\ref{thm-weil}(2), 
we obtain \eqref{eq:maskerr}.
\end{proof}

Let us end the section by illustrating  ${\bf b}_s(x)$ with some examples.

\begin{example}
    Suppose that $H =1$. Then
    $$
    {\bf b}_s(x) = {\bf b}_0(x) =  \sum_{t=0}^{r-1}p^t{\bf e}_t(x)+ \sum_{t=0}^{q-r-1} p^{r+t}{\bf f}_{t}(x).
    $$
    This further reduces back to our ${\mathsf b}(x)$ in $\R^1$ if $d = 1$ as well (see (\ref{eq:bdef})). If $d>1$, this becomes the vector-valued analogue of (\ref{eq:bdef}).
\end{example}

\begin{example}
    Suppose that $H =2$ and $d=3$, $r=1$, $q=2$. Then for $ x\in\mathbb{F}_{p^3} $
\[
 {\bf e}_0(x)=\begin{pmatrix}e_0(x)\\e_1(x)\\e_2(x)\end{pmatrix},\quad
 {\bf e}_1(x)=\begin{pmatrix}e_3(x)\\e_4(x)\\e_5(x)\end{pmatrix},\quad
 {\bf f}_0(x)=\begin{pmatrix}f_0(x)\\f_1(x)\\f_2(x)\end{pmatrix},\quad
 {\bf f}_1(x)=\begin{pmatrix}f_3(x)\\f_4(x)\\f_5(x)\end{pmatrix}.
\]
 {\color{black}
Since $\mathbf u_k=\mathbf e_k$ and $\mathbf v_k=\mathbf f_k$ for $k=0,1$, the spectral digit map is
$$
\mathbf b_s(x) = [\mathbf e_0(x)-\mathbf e_1(x)]_p +\mathbf e_1(x)
+p\bigl(\mathbf f_0(x)+\mathbf f_1(x)\bigr).
$$
}
\end{example}

\section{Salem and spectral properties in $\R^d$}\label{Sec8} 
 \subsection{The target Moran measure.} Fix integers $d,H\ge1$ and a real number $0<s\le d$. Choose integers $q_n,r_n$ satisfying
\begin{equation}\label{eq:parameters}
 2\le r_n<q_n,\qquad q_n\longrightarrow\infty,\qquad \frac{d r_n}{q_n}\longrightarrow s.
\end{equation}
To be more precise, for $s<d$, we can choose $q_n=n+\lceil 2d/s\rceil$ and $r_n=\lfloor sq_n/d\rfloor$; for $s=d$, we choose $q_n=n+2$ and $r_n=n+1$. Let
\[
 D_n=\left\lceil\frac{Hq_n}{r_n}\right\rceil.
\]
Let $p_0=1$. Select recursively $p_n$ as the {least odd prime} larger than $D_n$ and satisfying
\begin{equation}\label{eq:growth}
 p_n^{q_n}>(p_1^{q_1}\cdots p_{n-1}^{q_{n-1}})^{\,n},\qquad
\end{equation}
and
\begin{equation}\label{eq:small-error}
 (D_n-1)(2\log(2p_n^{r_n}))^d (2\log(2p_n))^{Hdq_n}\le p_n^{1/2}.
\end{equation}
As in Lemma~\ref{lem:primeexist}, both \eqref{eq:growth} and \eqref{eq:small-error} hold for every sufficiently large prime. Thus the sequence $\{p_n\}_n$ is well defined. For each $n\ge1$, let
$$
 \Psi_n: \mathbb F_{p_n^{dr_n}}\longrightarrow \mathbb F_{p_n}^{Hdq_n},
 \qquad
 \mathbf b_{n,s}: \mathbb F_{p_n^{dr_n}}\longrightarrow\Z^d
$$
be the maps defined in \eqref{eq:Psi-high} and \eqref{def-bs-map}, respectively, with $(p,q,r)=(p_n,q_n,r_n)$. We now use (\ref{def-bbs}) to define
\[
 \B_n = \B_{n,s}=\{\mathbf b_{n,s}(x):x\in\mathbb F_{p_n^{dr_n}}\}
 \subset\mathbb Z^d.
\]
We can also choose  \eqref{eq-u_k-rep} to form $\B_{n,0}$, but it may not be spectral and it may have repeated digits. In this section, we will only choose $\B_n$ to be $\B_{n,s}$ for the sake of proving our main theorem in the introduction, but readers can also construct Salem sets using the simpler digit set $\B_{n,0}$.

By Lemma \ref{lem:basic-property-highD}, $\B_n$ contains exactly one representative of each class modulo $p_n^{r_n}\mathbb Z^d$ and
\begin{equation}\label{eq:digit-properties}
 |\B_n|=p_n^{dr_n},\qquad
 \B_n\subset\{0,\ldots,H(p_n^{q_n}-1)\}^d.
\end{equation}
Set also 
\begin{equation}\label{def-Nn-Pn}
    N_n = p_n^{q_n}, \qquad P_n = p_1^{q_1}\cdots p_n^{q_n}.
\end{equation}

Define the Moran measure and its support by
\begin{equation}\label{eq:mu}
 \mu=\Conv_{n=1}^{\infty} \left(\frac1{|\B_n|}\sum_{b\in \B_n}\delta_{b/P_n}\right), \qquad  K=\left\{\sum_{n=1}^\infty\frac{b_n}{P_n}: b_n\in \B_n\right\} \subset[0,H]^d.
\end{equation}
We have $\operatorname{supp}\mu=K$ and
\begin{equation}\label{eq:Fourier-product}
 \widehat\mu(\xi)=\prod_{n=1}^\infty {\mathsf M}_{\B_n}(\xi/P_n),\qquad
 {\mathsf M}_{\B_n}(\xi)=\frac1{|\B_n|}\sum_{b\in \B_n}
 e^{-2\pi i b\cdot\xi}.
\end{equation}
We are now ready to state our main theorem analogous to its one-dimensional counterpart Theorem \ref{main-theorem-detail}, which also restates Theorem \ref{thm:main} in $\R^d$. 
\begin{theorem}\label{main-theorem-detail2}
    For each $0 < s\le d$, let $K$ and $\mu$ be the Moran set and measure defined in (\ref{eq:mu}).   Then 
    \begin{enumerate}
        \item For every
$0<\alpha<\min\{s/2,H\}$, we can find the constant $C_{s,d,H,\alpha}$ such that 
\begin{equation}\label{eq:decay}
 |\wh\mu(\xi)|\le C_{s,d,H,\alpha}(1+|\xi|)^{-\alpha},
 \qquad \xi\in\mathbb R^d.
\end{equation}
\item  Suppose that $H>s/2$. Then $K$ is a Salem set of dimension $s$. 
\item $\mu$ is a spectral measure with a spectrum 
\[
\Lambda=\bigcup_{n\ge1}\Lambda_n, \qquad \Lambda_n=L_1+P_1L_2+\cdots+P_{n-1}L_n,
\]
where 
\begin{equation}\label{def-Ln-high}
    L_n = p_n^{q_n-r_n}\left\{-\frac{p_n^{r_n}-1}{2},\cdots, -1,0,1, \cdots,\frac{p_n^{r_n}-1}{2}\right\}^d.
\end{equation}
    \end{enumerate}
\end{theorem}

\subsection{Salem property} The Fourier decay property follows by essentially
the same argument as in $\R^1$. Let 
\begin{equation}\label{def-eta_n}
 e_m(\xi) = \mathsf M_{\B_m}(\xi) - \prod_{j=1}^{d} \bigl(\mathsf D_{N_m}(\xi_j)\bigr)^H, \qquad
    \eta_m = \sup_{\xi\in\R^d}\abs{e_m(\xi)}.
\end{equation}

\begin{lemma}[Higher-dimensional version of Lemma~\ref{lem:scaleconseq}]
\label{lem:eta}
For every $0<\gamma<s/2$, there exists a constant $C_\gamma \ge 1$, depending only on $s,d,H,\gamma$, such that for $ n \ge 1 $,
\begin{equation}\label{eta-high}
 \eta_n\le C_\gamma N_n^{-\gamma}.
\end{equation}
Moreover, $\eta_n\to0$ and $N_n>P_{n-1}^{\,n}$.
\end{lemma}

\begin{proof}
Fix $0<\gamma<s/2$. For $n\in \mathbb{N}$, applying Theorem~\ref{thm:finite} with
$p=p_n$, $r=r_n$, $q=q_n$, and $D=D_n$, we obtain
$$
\begin{aligned}
    \eta_n \le& (D_n-1)p_n^{-dr_n/2} \bigl(2\log(2p_n^{r_n})\bigr)^d \bigl(2\log(2p_n)\bigr)^{Hdq_n}\\
    \le& p_n^{-dr_n/2} p_n^{1/2} \qquad \mbox{(by (\ref{eq:small-error}))}\\
    = & N_n^{-(dr_n-1)/(2q_n)} \qquad (\mbox{since} ~ N_n = p_n^{q_n}). 
\end{aligned}
$$
By \eqref{eq:parameters}, we have $q_n\to\infty$ and $dr_n/q_n\to s$. 
For $\gamma<s/2$, there exists an integer $n_\gamma\ge2$, such that for every $ n\ge n_\gamma $, $\frac{dr_n-1}{2q_n}\ge\gamma. $ 
Since $N_n>1$, for $ n\ge n_\gamma $,
$$
 \eta_n \le N_n^{-(dr_n-1)/(2q_n)} \le N_n^{-\gamma}.
$$
Therefore, \eqref{eta-high} holds by slightly enlarging the constant to include all $n<n_\gamma$. Moreover, $ N_n=p_n^{q_n}\ge3^{q_n} \to \infty. $ Since $\gamma>0$, by \eqref{eta-high}, $ \eta_n\le C_\gamma N_n^{-\gamma} \to 0. $
Finally, $N_n>P_{n-1}^{\,n}$ follows from (\ref{eq:growth}).
\end{proof}

\begin{proposition}[Vector analogue of Proposition \ref{prop:blockwise}]\label{prop:frequency-band}
Fix $0<\gamma<\min\{s/2,H\}$. Let $n\ge2$ and $P_{n-1}\le\norm{\xi}_\infty<P_n$.  Let
\[
 \mathbf t=\frac{\xi}{P_{n-1}},\qquad t=\norm{\mathbf t}_\infty\in[1,N_n).
\]
Then there exists a constant $C_{s,d,H,\gamma}>0$, independent of $n$ and $\xi$, such that
\begin{equation}\label{eq:vector-gamma}
{ |\widehat\mu(\xi)|\le C_{s,d,H,\gamma} \left(N_{n-1}^{-\gamma}t^{-H}+N_n^{-\gamma}\right), \quad 1\le t<N_n.}
\end{equation}
\end{proposition}
\begin{proof}
For convenience, for $ \mathbf t = (t_1,\cdots, t_d) \in \mathbb{R}^d $, write
$$ 
\mathsf D_N(\mathbf t)=\prod_{j=1}^{d}\mathsf D_N(t_j).
$$ 
As in the proof of Proposition~\ref{prop:blockwise}, 
\begin{equation}\label{eq:vector-start}
\begin{aligned}
 |\widehat\mu(\xi)|
 &\le |\mathsf M_{\B_{n-1}}(\mathbf t)| \left| \mathsf M_{\B_n}(\mathbf t/N_n) \mathsf M_{\B_{n+1}} \left(\frac{\mathbf t}{N_nN_{n+1}}\right) \right|\\
 &\le |\mathsf M_{\B_{n-1}}(\mathbf t)| \left( \left| \mathsf D_{N_nN_{n+1}}
 \left(\frac{\mathbf t}{N_nN_{n+1}}\right) \right|^H+\eta_n+\eta_{n+1} \right)\\
 &\le |\mathsf M_{\B_{n-1}}(\mathbf t)| \left| \mathsf D_{N_nN_{n+1}} \left(\frac{\mathbf t}{N_nN_{n+1}}\right) \right|^H +\eta_n+\eta_{n+1}\\
 &\le \bigl( |\mathsf D_{N_{n-1}}(\mathbf t)|^H+\eta_{n-1} \bigr) \left| \mathsf D_{N_nN_{n+1}} \left(\frac{\mathbf t}{N_nN_{n+1}}\right) \right|^H +\eta_n+\eta_{n+1}.
\end{aligned}
\end{equation}

Choose a coordinate $i$ with $|t_i|=t$. {We first suppose $ t \not\in \mathbb{Z} $.} Using the inequalities in \eqref{D_n-1,n,n+1} and the fact that $ \vert \mathsf{D}_N(\cdot) \vert $ is an even function,
\begin{equation}\label{Dirich-esti-1}
\begin{aligned} |\mathsf D_{N_{n-1}}(\mathbf t)|^H \left|\mathsf D_{N_nN_{n+1}} \left(\frac{\mathbf t}{N_nN_{n+1}}\right)\right|^H
 &\le \left( |\mathsf D_{N_{n-1}}(t_i)| \left|{\mathsf D}_{N_nN_{n+1}}\left(\frac{t_i}{N_nN_{n+1}}\right)\right|
 \right)^H\\
 &= \left( |\mathsf D_{N_{n-1}}(t)| \left|{\mathsf D}_{N_nN_{n+1}}\left(\frac{t}{N_nN_{n+1}}\right)\right|
 \right)^H \\
 &\le\left(\frac{\pi}{4N_{n-1}t}\right)^H.
\end{aligned}
\end{equation}
Since $1\le t<N_n\le N_nN_{n+1}/2$, by Lemma~\ref{lem:Dbounds},
\begin{equation}\label{Dirich-esti-2}
    \left| \mathsf D_{N_nN_{n+1}} \left(\frac{\mathbf t}{N_nN_{n+1}}\right) \right|^H \le \left| \mathsf D_{N_nN_{n+1}} \left(\frac{t_i}{N_nN_{n+1}}\right) \right|^H  = \left| \mathsf D_{N_nN_{n+1}} \left(\frac{t}{N_nN_{n+1}}\right) \right|^H \le \left(\frac{\pi}{4t}\right)^H.
\end{equation}
{\color{black}The preceding two estimates \eqref{Dirich-esti-1} and \eqref{Dirich-esti-2} also hold when \(t\in\mathbb Z\), since the  Dirichlet factor $ \mathsf D_{N_nN_{n+1}} \left(t/(N_nN_{n+1}) \right) $ vanishes. Substituting these bounds into \eqref{eq:vector-start}, we obtain
\[
\begin{aligned}
|\widehat\mu(\xi)|
&\le \left(\frac{\pi}{4}\right)^H \bigl(N_{n-1}^{-H}+\eta_{n-1}\bigr)t^{-H} +\eta_n+\eta_{n+1}\\
&\le \left(\frac{\pi}{4}\right)^H(1+C_\gamma)N_{n-1}^{-\gamma}t^{-H}
+2C_\gamma N_n^{-\gamma}\\
&\le 2C_\gamma\left(N_{n-1}^{-\gamma}t^{-H}+N_n^{-\gamma}\right).
\end{aligned}
\]
Here the second inequality follows from Lemma~\ref{lem:eta}, the assumption \(\gamma<H\), and the condition
\[
N_{n+1}>P_n^{\,n+1}>N_n.
\]
The last inequality uses $C_\gamma\ge1$ and $(\pi/4)^H\le1$. Therefore, \eqref{eq:vector-gamma} holds with $C_{s,d,H,\gamma}=2C_\gamma$.}
\end{proof}

\begin{proof}[Proof of Theorem \ref{main-theorem-detail2} (1), (2)]
We follow the proof of Theorem~\ref{main-theorem-detail} (1) in Section~\ref{Sec5}. Choose $\alpha<\gamma<\min\{s/2,H\}$, and take $n_0\ge3$ sufficiently large such that for $ n \ge n_0 $,
\begin{equation}\label{eq:choice-of-large-n-high}
\frac{\gamma(n-1)}{n}\ge\alpha, \qquad n\left(\gamma-\alpha\right)\ge\alpha.
\end{equation}

Take $ \xi \in \mathbb{R}^d $ with $\norm{\xi}_\infty\ge P_{n_0-1}$. Choose
$n\ge n_0$ such that
\begin{equation}\label{xi-Pn}
    P_{n-1}\le\norm{\xi}_\infty<P_n.
\end{equation}
Let $ t=\frac{\norm{\xi}_\infty}{P_{n-1}}. $ Then $1\le t<N_n$. Repeating the derivation of \eqref{eq:Nnminusone-separation} with \eqref{eq:growth}, and using the first inequality in \eqref{eq:choice-of-large-n-high}, we obtain
\begin{equation}\label{8.13-term1}
    N_{n-1}^{-\gamma}t^{-H} \le P_{n-1}^{-\gamma(n-1)/n}t^{-H} \le P_{n-1}^{-\alpha}t^{-\alpha} = \norm{\xi}_\infty^{-\alpha},
\end{equation}
where we also used $t\ge1$ and $H>\alpha$. Moreover, by \eqref{eq:growth} and the second inequality in \eqref{eq:choice-of-large-n-high},
\begin{equation}\label{Nngamma-alpha}
    N_n^{\gamma-\alpha} >P_{n-1}^{n(\gamma-\alpha)}\ge P_{n-1}^{\alpha}.
\end{equation}
Consequently, combining \eqref{xi-Pn} and \eqref{Nngamma-alpha}, we have
\begin{equation}\label{8.13-term2}
    N_n^{-\gamma} =N_n^{-(\gamma-\alpha)}N_n^{-\alpha} \le P_{n-1}^{-\alpha}N_n^{-\alpha} =P_n^{-\alpha} \le\norm{\xi}_\infty^{-\alpha}.
\end{equation}
Applying Proposition \ref{prop:frequency-band}, together with \eqref{8.13-term1} and \eqref{8.13-term2}, yields
\begin{equation}\label{mu-Cs}
    |\wh\mu(\xi)| \le 2C_{s,d,H,\gamma}\norm{\xi}_\infty^{-\alpha} \qquad\textup{for }\norm{\xi}_\infty\ge P_{n_0-1}.
\end{equation}
Using $|\widehat\mu(\xi)|\le1$ and $|\xi|\le\sqrt d\,\|\xi\|_\infty$, we obtain \eqref{eq:decay} after enlarging the constant in the inequality \eqref{mu-Cs}. Finally, for part (2), if $H>s/2$, we know that $s\le \dim_F(K)$.  The covering argument in the proof of Theorem~\ref{main-theorem-detail}(2) still applies: $K$ is covered by at most $\prod_{j=1}^n p_j^{dr_j}$ cubes of diameter $H\sqrt d/P_n$. Following the same methods in the proof of Theorem~\ref{main-theorem-detail}(2), it can be shown $\dim_{\textup{H}} K\le s$ {and $ \mathcal{L}^d(K) = 0 $ }. To avoid too much repetition, we will omit this calculation. Hence, $K$ is a Salem set. 
\end{proof}

\section{Spectral properties in arbitrary dimension}\label{Sec9} 
This section is devoted to proving the spectrality of $\mu_s$, namely part~(3) of Theorem~\ref{main-theorem-detail} and Theorem \ref{main-theorem-detail2}. 
We prove part~(3) of Theorem~\ref{main-theorem-detail2}. The same argument applies to the one-dimensional construction with $d=H=1$ and proves part~(3) of Theorem~\ref{main-theorem-detail}. Throughout this section, fix $0<s\le d$.
Recall that $q_n,r_n$ are defined in \eqref{eq:parameters}, and that $D_n = \lceil Hq_n/r_n\rceil$. For each $n\ge1$, the prime $p_n$ is chosen recursively
as the least odd prime $p>D_n$ satisfying
\eqref{eq:growth} and \eqref{eq:small-error}. We set
\[
N_n=p_n^{q_n},\qquad M_n=p_n^{dr_n},\qquad
P_n=N_1\cdots N_n,
\]
and let $\mathcal B_n$ be the digit set defined in
\eqref{def-bbs}, with $|\mathcal B_n|=M_n$.
All these parameters and digit sets depend on the fixed
value of $s$. We omit this dependence in the notation. We denote by $\mu_s$ the associated equal-weighted Cantor--Moran measure defined by \eqref{cantor-moran_measure}. Recall from \eqref{def-Ln-high}, we have
$$
L_n = p_n^{q_n-r_n}\left\{-\frac{p_n^{r_n}-1}{2},-\frac{p_n^{r_n}-1}{2}+1,\cdots, -1,0,1, \cdots,\frac{p_n^{r_n}-1}{2}\right\}^d,
$$
for $ n \in \mathbb{N} $.
For convenience, set
\[
\mathcal K_n
=\left\{-\frac{p_n^{r_n}-1}{2},-\frac{p_n^{r_n}-1}{2}+1,\cdots, -1,0,1, \cdots,\frac{p_n^{r_n}-1}{2}\right\}^d,
\]
so that $L_n=p_n^{q_n-r_n}\mathcal K_n$.
Define
\[
 \Lambda_n=L_1+P_1L_2+\cdots+P_{n-1}L_n,\qquad \Lambda=\bigcup_{n\ge1}\Lambda_n.
\]

\begin{proposition}\label{prop:block}
For each $n\ge1$, let $N_n=p_n^{q_n}$ as in \eqref{def-Nn-Pn}. Then the matrix 
\[
H_n: = \frac1{\sqrt{p_n^{dr_n}}}\left(e^{2\pi i\frac{{\bf b} \cdot \boldsymbol{\ell}}{N_n}}\right)_{{\bf b}\in \B_n,\, \boldsymbol{\ell} \in L_n}
\]
is unitary, that is, $H_n^{\ast}H_n=I$, where $H_n^{\ast}$ denotes the conjugate transpose of $H_n$.
\end{proposition} 
\begin{proof}
By Lemma~\ref{lem:basic-property-highD}, $\B_n$ contains exactly one representative of each residue class modulo $p_n^{r_n}\Z^d$. For $\mathbf b=(b_1,\ldots,b_d)\in\B_n$ and $\boldsymbol{\ell}=p_n^{q_n-r_n}\mathbf k\in L_n$, where $\mathbf k=(k_1,\ldots,k_d)\in\mathcal K_n$, we have
$$
 e^{2\pi i\mathbf b\cdot \boldsymbol{\ell} /N_n} = e^{2\pi i\mathbf b\cdot\mathbf k/p_n^{r_n}} = \prod_{s=1}^{d} e^{2\pi i[b_s]_{p_n^{r_n}}k_s/p_n^{r_n}}.
$$
For $\boldsymbol{\ell}=p_n^{q_n-r_n}\mathbf k$, $\boldsymbol{\ell}'=p_n^{q_n-r_n}\mathbf k' \in L_n$, since $\B_n$ contains exactly one representative of each residue class modulo $p_n^{r_n}\Z^d$,
\begin{align*}
 &\frac1{|\B_n|}\sum_{\mathbf b\in\B_n} e^{2\pi i\mathbf b\cdot (\boldsymbol{\ell}'-\boldsymbol{\ell})/N_n}\\
 = & \frac1{p_n^{dr_n}} \sum_{\mathbf a = (a_1, \cdots, a_d) \in\{0,\ldots,p_n^{r_n}-1\}^d} e^{2\pi i\sum_{s=1}^{d}a_s(k'_s-k_s)/p_n^{r_n}}\\
 = & \prod_{s=1}^{d} \left( \frac1{p_n^{r_n}} \sum_{a=0}^{p_n^{r_n}-1} e^{2\pi i a(k'_s-k_s)/p_n^{r_n}} \right)
 = \begin{cases}
 1,&\boldsymbol{\ell}=\boldsymbol{\ell}',\\
 0,&\boldsymbol{\ell}\ne\boldsymbol{\ell}'.
 \end{cases}
\end{align*}
Therefore, since $|\B_n|=|L_n|$, the matrix
$$
 \frac1{\sqrt{|\B_n|}}  \left(e^{2\pi i\mathbf b\cdot\boldsymbol{\ell}/N_n}\right) _{\mathbf b\in\B_n,\,\boldsymbol{\ell}\in L_n}
$$
has orthonormal columns and is unitary.
\end{proof}

Recall that we say that  $(N_n, \B_n, L_n)$ forms a Hadamard triple if the matrix $H_n$ is unitary. Given a system of Hadamard triples $(N_n,\B_n,L_n)$, it is well-known that 
$$
\Lambda_n=L_1+N_1L_2+\cdots+(N_1\cdots N_{n-1})L_n
$$
is a spectrum for the measure of the first $n$ convolutions of the Cantor-Moran measure
$$
\nu_n = \Conv_{ k = 1 }^{ n } \left( \sum_{ b \in \mathcal{B}_k } \dfrac{1}{ \#\B_k }\delta_{ b \cdot ( N_1 \cdots N_k )^{-1} } \right).
$$
We now write our Cantor-Moran measure as
$$
\mu_s = \nu_n\ast\nu_{>n},
$$
where $\nu_{>n}$ is the rest of the convolutions after the first $n$ terms. To pass the spectrum to the limit measure, we use the following theorem from Lai-Wang \cite[Theorem~3.3]{LaiWang}. Indeed, the same proof also works for high dimensions. Readers can check \cite{LiWang2024} for high dimensions.  Indeed, an early version of this theorem was first observed by Strichartz \cite{Strichartz2000MockFourier} and later formulated rigorously in \cite{DHL2019} for self-affine measures. The following theorem generalizes naturally to Moran measures and it also works more generally to produce Fourier frames, but we state only the special case for the sake of our proof.  

\begin{theorem}\label{LaiWang}
Let $(N_n,B_n,L_n)$ be a system of Hadamard triples in $\R^d$, let $\mu$ be the associated infinite convolution, and let
\[
\Lambda_n=L_1+N_1L_2+\cdots+(N_1\cdots N_{n-1})L_n,
\qquad
\Lambda=\bigcup_{n\ge1}\Lambda_n.
\]
Suppose that for some $n_0\in \N$, 
\begin{equation}
\delta_{n_0}(\Lambda)
:=\inf_{n\ge n_0}\inf_{\lambda\in\Lambda_n}
\abs{\widehat{\nu_{>n}}(\lambda)}^2
>0.
\label{eq:delta-def}
\end{equation}
Then $E(\Lambda)=\{e^{2\pi i\lambda\cdot x}:\lambda\in\Lambda\}$ is an orthonormal basis for $L^2(\mu)$.
\end{theorem}

 We only need to directly check that \eqref{eq:delta-def} holds for our $\nu_{>n}$ and $\Lambda_n$. Fix $\boldsymbol{\lambda}=(\lambda_1,\ldots,\lambda_d)\in\Lambda_n$ and write
$$
 \boldsymbol{\lambda} =\sum_{j=1}^{n}P_{j-1}\boldsymbol{\ell}_j,
$$
where $ \boldsymbol{\ell}_j\in L_j $ for $ 1 \le j \le n $. For $ n \ge 1 $, $ 1 \le j \le n $, let $ \boldsymbol{\ell}_j = (\ell_{j,1}, \cdots, \ell_{j,d}) \in L_j $. By the definition of $L_j$, for $1\le j\le n$ and $1\le i\le d$, we have
\begin{equation}\label{ellj-upp}
|\ell_{j,i}|\le p_j^{q_j-r_j}\frac{p_j^{r_j}-1}{2}=\frac{N_j-p_j^{q_j-r_j}}{2}
\le\frac{N_j-1}{2}.
\end{equation}

Since $P_j=P_{j-1}N_j$ and $P_0=1$, we obtain,
for each $1\le i\le d$,
\begin{align*}
 |\lambda_i|
 &\le\sum_{j=1}^{n}P_{j-1}|\ell_{j,i}|\\
 &\le\frac12\sum_{j=1}^{n}P_{j-1}(N_j-1)\\
 &=\frac12\sum_{j=1}^{n}(P_j-P_{j-1})
 =\frac{P_n-1}{2}.
\end{align*}
Taking the maximum over $1\le i\le d$ and dividing by $P_n$, we obtain
\begin{equation}\label{eq:lambda-bound}
 \norm{\boldsymbol{\lambda}/P_n}_\infty \le\frac{P_n-1}{2P_n} <\frac12.
\end{equation}

We next estimate the Dirichlet masks and give a universal estimate for the later tail factor.

\begin{lemma}\label{DN(x/N)}
Let $ N $ be an integer with $ N \ge 2 $. For every $x \in \mathbb{R} $ with $\abs{x}\le1/2$, we have
\begin{equation}\label{eq:dirichlet-lower}
\abs{{\mathsf D}_N(x/N)}\ge\frac2\pi.
\end{equation}
\end{lemma}

\begin{proof}
The lemma is trivial when $x =0$ where ${\mathsf D}_N(0)=1$. For $x\ne0$, by using the second equality of \eqref{eq:dirichlet} and the fact that $\abs{\sin y}\le\abs{y}$ holds for all $ y \in \mathbb{R} $, we have
\[
\abs{{\mathsf D}_N(x/N)}
=\frac{\abs{\sin(\pi x)}}{N\abs{\sin(\pi x/N)}}\ge \frac{\abs{\sin(\pi x)}}{\pi\abs{x}}.
\]
Note that the function $ x \mapsto \sin(\pi x)/(\pi x)$ is decreasing when $ x \in [0,1/2] $, so its minimum is $2/\pi$ when $ x = 1/2 $. Therefore, \eqref{eq:dirichlet-lower} holds.
\end{proof}


\begin{lemma}\label{lem:large-digit-tail}
Let $N,H,d\ge1$ be integers, and let $\mathcal{B}\subset\{0,\ldots,H(N-1)\}^d$ be nonempty. Then, for every $\mathbf z\in\R^d$,
\begin{equation}\label{eq:lip}
 |\mathsf M_\mathcal{B}(\mathbf z/N)-1| \le2\pi Hd\norm{\mathbf z}_\infty.
\end{equation}
\end{lemma}

\begin{proof}
By the definition of $\mathsf M_\mathcal{B}$ and the triangle inequality,
\begin{equation}\label{lem-8.5-eq1}
    |\mathsf M_\mathcal{B}(\mathbf z/N)-1| \le \frac1{|\mathcal{B}|}\sum_{\mathbf b\in \mathcal{B}} \left|e^{-2\pi i\mathbf b\cdot\mathbf z/N}-1\right|.
\end{equation}
Write $ \mathbf b = (b_1, \cdots, b_d) $ and $ \mathbf z = (z_1, \cdots, z_d) $. Using $|e^{iu}-1|\le|u|$ for all $ u \in \mathbb{R} $, and $|\mathbf b\cdot\mathbf z| \le\sum_{i=1}^{d}|b_i||z_i|$,  we obtain
\begin{equation}\label{lem-8.5-eq2}
    |\mathsf M_\mathcal{B}(\mathbf z/N)-1| \le  \frac{2\pi}{N|\mathcal{B}|} \sum_{\mathbf b\in \mathcal{B}} |\mathbf b\cdot\mathbf z| \le \frac{2\pi}{N|\mathcal{B}|} \sum_{\mathbf b\in \mathcal{B}}\sum_{i=1}^{d}|b_i||z_i|.
\end{equation}
For $ \mathbf b = (b_1, \cdots, b_d) \in \mathcal{B} $, since $0\le b_i\le H(N-1)$ for $ 1 \le i \le d $,
the last term in \eqref{lem-8.5-eq2} is at most
$$
 \frac{2\pi H(N-1)}{N}\sum_{i=1}^{d}|z_i| \le2\pi H d\norm{\mathbf z}_\infty.
$$
This proves \eqref{eq:lip}.
\end{proof}


\begin{proof}[Proof of Theorem \ref{main-theorem-detail2} (3).] 
{Choose an integer $N_*>2\pi Hd$.
Since $\eta_j\to0$ by Lemma~\ref{lem:eta} and $N_j\to\infty$ by \eqref{eq:growth}, take $n_0\ge1$ sufficiently large such that for $ j\ge n_0+1 $,
$$
 \eta_j\le\frac12\left(\frac2\pi\right)^{Hd} \textup{ and } N_j\ge N_*.
$$
By Theorem~\ref{LaiWang}, it suffices to prove that there exists $c_*>0$ such that
\begin{equation}\label{eq:tail-lower-highD}
 \inf_{n\ge n_0} \inf_{\boldsymbol{\lambda}\in\Lambda_n} |\wh{\nu_{>n}}(\boldsymbol{\lambda})| \ge c_*.
\end{equation}
Fix $n\ge n_0$ and $\boldsymbol{\lambda}\in\Lambda_n$, and let $\mathbf t=\boldsymbol{\lambda}/P_n$. By \eqref{eq:lambda-bound}, $\norm{\mathbf t}_\infty<1/2$.}


For each $1\le i\le d$, we have $|t_i|\le\norm{\mathbf t}_\infty<1/2$.
Applying Lemma~\ref{DN(x/N)} with $N=N_{n+1}$ and $x=t_i$, we obtain
\begin{equation}\label{prop-8.6}
    \left|\mathsf D_{N_{n+1}}(t_i/N_{n+1})\right| \ge\frac2\pi.
\end{equation}
Therefore, following the definition of $ \eta_n $ in \eqref{def-eta_n}, and the triangle inequality, we have
\begin{equation}\label{prop8.6-eq1}
\begin{aligned}
 \left|\mathsf M_{\B_{n+1}}(\mathbf t/N_{n+1})\right|
 &\ge \prod_{i=1}^{d} \left|\mathsf D_{N_{n+1}}(t_i/N_{n+1})\right|^H
 -\eta_{n+1}\\
 &\ge \left(\frac2\pi\right)^{Hd}-\eta_{n+1}\\
 &\ge \frac12\left(\frac2\pi\right)^{Hd},
\end{aligned}
\end{equation}
where the last inequality follows from the choice of $n_0$. For $r\ge1$, let
$$
 \mathbf z_r= \frac{\mathbf t}{N_{n+1}\cdots N_{n+r}}.
$$

Since $n\ge n_0$, according to the choice of $n_0$, we have $N_{n+j}\ge N_*$ for $1\le j\le r$. Moreover, since $\norm{\mathbf t}_\infty<1/2$, 
$$
 \norm{\mathbf z_r}_\infty = \frac{\norm{\mathbf t}_\infty} {N_{n+1}\cdots N_{n+r}} \le \frac{1/2}{N_*^r} = \frac1{2N_*^r}.
$$
Therefore, by \eqref{eq:lip} and the triangle inequality,
\begin{equation}\label{prop8.6-eq2}
    \left| \mathsf M_{\B_{n+r+1}} \left(\frac{\mathbf z_r}{N_{n+r+1}}\right) \right| \ge1-2\pi Hd\norm{\mathbf z_r}_\infty \ge1-\frac{\pi Hd}{N_*^r}.
\end{equation}
Combining the estimates \eqref{prop8.6-eq1} and \eqref{prop8.6-eq2}, for $ n \ge n_0 $,
\begin{align*}
 |\wh{\nu_{>n}}(\boldsymbol{\lambda})| &= \left|\mathsf M_{\B_{n+1}}(\mathbf t/N_{n+1})\right| \prod_{r=1}^{\infty} \left| \mathsf M_{\B_{n+r+1}}
 \left(\frac{\mathbf z_r}{N_{n+r+1}}\right) \right|\\
 &\ge \frac12\left(\frac2\pi\right)^{Hd} \prod_{r=1}^{\infty} \left(1-\frac{\pi Hd}{N_*^r}\right) =:c_*.
\end{align*}
Each factor in the infinite product above is positive and $\sum_{r\ge1}\pi Hd/N_*^r<\infty$. Therefore, $c_*>0$. This constant is independent of $n$ and
$\boldsymbol{\lambda}$, proving \eqref{eq:tail-lower-highD} and hence Theorem \ref{main-theorem-detail2} (3).
\end{proof}

\bibliographystyle{alpha}
\bibliography{universal_bib}

\end{document}